\documentclass[10pt, leqno]{amsart}
\usepackage[
hmarginratio={1:1},     
vmarginratio={1:1},     
textwidth=400pt,        
heightrounded,          
]{geometry}

\usepackage{latexsym,amssymb,bbm, hyperref}
\usepackage[hang,flushmargin,symbol*]{footmisc} 
\usepackage{etoolbox,graphicx,subcaption,color,soul}

\makeatletter
\patchcmd{\@maketitle}
  {\ifx\@empty\@dedicatory}
  {\ifx\@empty\@date \else {\vskip3ex \centering\footnotesize\@date\par\vskip1ex}\fi
   \ifx\@empty\@dedicatory}
  {}{}
\patchcmd{\@adminfootnotes}
  {\ifx\@empty\@date\else \@footnotetext{\@setdate}\fi}
  {}{}{}
  
  \def\blfootnote{\gdef\@thefnmark{}\@footnotetext}
\makeatother

\newcommand{\diam}{\mathrm{diam}\,}
\newcommand{\dist}{\mathrm{dist}}

\newcommand{\G}{\overline{g}^q_{s,d}}
\newcommand{\g}{\underline{g}^q_{s,d}}
\newcommand{\qc}{q_{ \scalebox{0.6}{$ C $}}}
\newcommand{\cl}{\mathcal{R}_s^d }
\newcommand{\e}{E^q_{s,d}}
\newcommand{\n}{\mathcal{N}}
\newcommand{\N}{\mathbb{N}}

\newcommand{\R}{\mathbb{R}}
\newcommand{\E}{\mathcal{E}_{s,d}^q}
\newcommand{\supp}{\operatorname{supp}}
\newcommand{\T}{\tau_{s,d}}
\newcommand{\Q}{\overline{q}_\epsilon}
\newcommand{\q}{\underline{q}_\epsilon}
\newcommand{\B}{B_{j}}
\newcommand{\slim}{\lim_{\substack{N\to \infty \\ N\in\mathcal{N}  }}}

\renewcommand{\c}{\mathcal{C}}
\renewcommand{\d}{\,\mathrm{d}}
\renewcommand{\H}{\mathcal{H}_d^A }
\newcommand{\h}{\mathcal{H}_d}
\renewcommand{\L}{L_*}
\renewcommand{\qc}{q_{ \scalebox{0.6}{\ensuremath{C}}}}
\renewcommand{\c}{C_{s,d}}

\newtheorem{theorem}{Theorem}[section]
\newtheorem{lemma}[theorem]{Lemma}

\newtheorem{proposition}[theorem]{Proposition}

\newtheorem{thm}{Theorem}
\newtheorem{corollary}[theorem]{Corollary}
\theoremstyle{definition}
\newtheorem{remark}[theorem]{Remark}
\newtheorem{example}[theorem]{Example}

\numberwithin{equation}{section}

\title[Point Configurations via Hypersingular Riesz Energy]{Generating Point Configurations via Hypersingular Riesz Energy With an External Field}
\author{D. P. Hardin}
\author{E. B. Saff}
\author[O. V. Vlasiuk]{O. V. Vlasiuk$ ^\dagger $ }

\address{Center for Constructive Approximation} 
\address{Department of Mathematics, Vanderbilt University, Nashville, TN, 37240}
\email{doug.hardin@vanderbilt.edu}
\email{edward.b.saff@vanderbilt.edu}
\email{oleksandr.vlasiuk@vanderbilt.edu}

\begin{document}

\begin{abstract}
	For a compact $ d $-dimensional rectifiable subset of $ \mathbb R^{p} $  we study asymptotic properties as  $ N\to\infty $ of $N$-point configurations minimizing the energy arising from a Riesz $ s $-potential $ 1/r^s $ and an external field in the hypersingular case $ s\geq d$. Formulas for the weak$ ^* $ limit of normalized counting measures of such optimal point sets and the first-order asymptotic values of minimal energy are obtained. As an application, we derive a method for generating configurations whose normalized counting measures converge to a given absolutely continuous measure supported on a rectifiable subset of $ \mathbb R^{p} $. Results on separation and covering properties of discrete minimizers are given. Our theorems are illustrated with several numerical examples.\\ 
\end{abstract}
\date{\today}
\maketitle
\section{Introduction}

\blfootnote{ The research of the authors was supported, in part, by National Science Foundation grants DMS-1412428 and DMS-1516400. A part of this work was conducted during the Minimal Energy Point Sets, Lattices and Designs Workshop at the Erwin Schr\"odinger International Institute for Mathematical Physics.  \\
$ ^\dagger $The research of this author was completed as a part of a Ph.D. dissertation.}

We are concerned with the problem of minimizing the discrete Riesz $ s $-energy of $ N $ particles constrained to a compact subset $ A $ of $ \R^p  $ of Hausdorff dimension $ d $ under the influence of an external field $ q(x) $. More precisely, we minimize
\begin{equation}\label{s_energy}
\begin{aligned}
E^q_{s,d}(\omega_N):=  \sum_{ \substack{x\neq y \\ x,y\in\omega_N} }  
 |x-y|^{-s}  + \frac{\T(N)}{N }\,\sum_{x\in\omega_N} q(x), \quad s\geq d,    
\end{aligned}
\end{equation}
for
\begin{equation}
\T(N):=   \begin{cases} N^{1+s/d},\ &s>d,  \\ N^2\log N,\ &s=d, \end{cases} 
\end{equation}
over $ N $-element subsets $ \omega_N \subset A $.

The factor $ \T(N)/N $ is chosen so that the two terms on the right hand side of \eqref{s_energy} have the same order of growth as $ N\to\infty $. Here we consider only the case when $s $ is chosen greater than or equal to the dimension of the set $ A $ because for $ s<d $ such external field problems come under the umbrella of classical potential theory and have been well studied as we describe below.

One motivation to consider this energy expression is that (under mild conditions on the set $A $) for any probability measure $ \mu $ on $ A $ that is absolutely continuous with respect to the $ d $-dimensional Hausdorff measure restricted to $ A $, there is an easily described external field $ q(x) $ for which the normalized counting measures of the minimizers of \eqref{s_energy}  weak$ ^* $ converge to $ \mu $ (formal definitions are given in the next two subsections).

For $ s > d $ the minimizers of \eqref{s_energy} are shown to have optimal orders of separation and covering as $ N\to \infty $. Minimization of \eqref{s_energy} therefore provides well-distributed nodes on compact sets, which can be used for a number of applications; for example, meshless  methods \cite{nguyen2008meshless}, halftoning \cite{wu2014riesz} and sensor deployment \cite{johnson2012more}.

External fields arise in the \textit{Gauss variational problem}, which involves minimizing the functional 
\begin{equation}\label{gauss-variation}
I_{\kappa,q}(\mu)  :=\iint\limits_{A\times A}  \kappa(x,y) \d\mu(x) \d\mu(y) + 2 \int_A q(x) \d \mu(x),
\end{equation}
for a pair of fixed integrable lower semi-continuous functions $ \kappa: A\times A \to {\R\cup \{+\infty \} } ,\ q: A \to \R\cup \{+\infty \}  $ over the probability measures supported on $ A $. The classical work of Ohtsuka \cite{ohtsuka1961potentials} deals with this question when $ A $ is locally compact. The case $ \kappa(x,y) = \log \frac{1}{|x-y|}  $ and a number of its applications to constructive analysis are extensively treated in the book \cite{saff2013logarithmic} by Saff and Totik. More recently, the question of solvability of the Gauss variational problem was considered by Zorii in \cite{zorii2003equilibrium} and \cite{zorii2005necessary}.

The discrete form of \eqref{gauss-variation} is the problem of minimization over all $ N $-element sets $   \omega_N \subset A  $ of 
\begin{equation}\label{discrete-general}
 E_{\kappa,q} (\omega_N)=\sum_{ \substack{x\neq y \\ x,y\in\omega_N} } \kappa(x,y) + N \sum_{x\in \omega_N} q(x), \qquad N=1,2,\ldots. 
\end{equation}
Such problems have been studied by Petrache, Rougerie and Serfaty in \cite{petrache2014next}, \cite{rougerie2015higher} for the \textit{Riesz $s$-kernel}   
\begin{equation}
 \kappa_s(x,y) :=  |x-y|^{-s},  
\end{equation} 
with $ s<d $, and for the logarithmic kernel in \cite{saff2013logarithmic}. An earlier series of papers \cite{Dragnev2007}, \cite{Brauchart2009} and \cite{Brauchart2014}  by Brauchart et al. explores minima of \eqref{discrete-general} when $ A $ is a $ d $-dimensional sphere and $ d-2\leq s < d $. The paper \cite{biloglia2015weighted} by Bilogliadov considers minimizing \eqref{gauss-variation} with the Riesz kernel for $ s=1 $ and a rotationally symmetric $ q $ over probability measures supported on the $ 2 $-dimensional unit sphere.

\subsection{Hypersingular Riesz kernels} 

We call the Riesz kernels $ \kappa_s(x,y) = |x-y|^{-s} $ \textit{hypersingular} when $ s\geq d = \dim A $, and deal only with this case from now on. 
The reason to consider such kernels will become evident from the following result, which shows that the minimizers  of \eqref{s_energy} with $ q\equiv0 $ are well-distributed on the set $ A $, which need not to be the case when $ s<d $.

For the purposes of studying the asymptotic behavior of $ N $-point configurations $\omega_N$ on $A$ we consider their normalized counting measures. Recall that such measures  $N^{-1} \sum_{x\in \omega_N }  \delta_x  $ are said to   \textit{weak$ ^* $ converge} to the measure $ \lambda $ if
\begin{equation}\label{weak_convergence}
 \forall f \in C(A),\qquad  \frac{1}{N} \sum_{x\in\omega_N } f(x)  {\longrightarrow} \int f \d \lambda \quad \text{ as } N\to \infty, 
\end{equation}
where, as usual, $ C(A) $ denotes the family of all continuous functions defined on $ A $. Weak$ ^* $ convergence is denoted by $ \stackrel{*}{\longrightarrow} $.

We further need to impose some regularity conditions on the underlying compact set.  A set $ A\subset\R^{p}  $ is said to be  $ d $\textit{-rectifiable} if it is the image of a bounded subset of $ \R^d $ under a Lipschitz mapping. 
Note that any subset of a $ d $-rectifiable set is also $ d $-rectifiable. Here and below we write  $ \mathbb{B }^d $ for the $ d $-dimensional unit ball. We use $\h $ to denote the $ d $-dimensional Hausdorff measure on  $ \R^{p} $ normalized so that $ [0,1]^d\subset \R^d\subset \R^p $ has unit volume, and by $ \H $ its restriction to $ A $. In particular, for a $ d $-rectifiable $ A $, $ \h(A)<\infty $.
The following theorem concerns a variant  of \eqref{s_energy} without external field:
\begin{equation}\label{intro_minimal_energy}
\mathcal{E}_s(A,N):=\inf_{\omega_N\subset A} E_s(\omega_N), 
\end{equation}
where $ E_s(\omega_N) : = \sum_{x\neq y\in \omega_N} \kappa_s(x,y) $. This infimum is attained for compact sets $ A $, because the Riesz $ s $-kernel is lower semi-continuous on $ A\times A $.

\begin{thm}[\cite{HaSa2005},\cite{borodachov2008asymptotics}]\label{poppy seed}
Suppose $ s \geq d $ and $ A\subset \R^{p} $ is $ d $-rectifiable and compact. If $ s = d $, it is further assumed that $ A $ is a subset of a $ d $-dimensional $ C^1 $ manifold. Then for $ s = d $ 
\begin{equation}\label{unweighted_s=d}
\lim_{N\to\infty}  \frac{\mathcal{E}_s(A,N)}{N^2\log N}= \frac{\mathcal{H}_d (\mathbb{B }^d  )}{\mathcal{H}_d(A)},
\end{equation}
while for $ s > d $, the following limit  exists:

\begin{equation}\label{unweighted_s>d}
\lim_{N\to\infty}  \frac{\mathcal{E}_s(A,N)}{N^{1+s/d}}= \frac{\c  }{\mathcal{H}_d(A)^{s/d}},
\end{equation}
where $ \c $ is a finite positive constant independent of $ A $ and $ p $, and $ 1/0 = +\infty $.
Furthermore, if $ \mathcal{H}_d(A) > 0 $ and $\{\omega_N\}_{N\ge 2}$  is any sequence of $ N $-point configurations on $A$ satisfying  
\begin{equation}
\lim_{N\to\infty}\frac{E_s(\omega_N)}{\mathcal{E}_s(A,N)}=1,
\end{equation}
then 
\begin{equation}
\frac{1}{N}\sum_{x\in {\omega}_N} \delta_{ x}\stackrel{*}{\longrightarrow} \frac{\d \H }{\mathcal{H}_d(A) }, \quad N\to\infty .
\end{equation}


\end{thm}
This theorem is sometimes described as the \textit{Poppy-seed bagel theorem}, a name that alludes to discrete equilibrium configurations on the torus. It first appeared in \cite[Theorem 2.1]{HaSa2005}, and in the present generality in \cite[Theorems 1--3]{borodachov2008asymptotics}.

In particular, the theorem holds for any compact  $ A \subset \R^d $ as well as any compact subset of a smooth $ d $-dimensional manifold.
To be consistent with the notation of \eqref{unweighted_s>d},  we define $ C_{d,d} $ according to \eqref{unweighted_s=d}: 
\begin{equation}\label{c_dd}
 C_{d,d}:= {\mathcal{H}_d(\mathbb{B}^d)}= \frac{\pi^{d/2} }{\Gamma\left(d/2+1 \right) } , \quad d\geq1, 
\end{equation}
where $ \Gamma $ is the standard gamma function.
It is known for $ d=1,\ s> 1 $ that 
\begin{equation}\label{cs1}
 \begin{aligned}
 C_{s,1}= 2\zeta(s),&\quad s>1,
 \end{aligned} 
\end{equation}
where $ \zeta $ is the Riemann zeta function, see e.g. \cite{FiMaRaSa2004}. However, for dimensions $ d \geq 2 $ the exact value of $ \c $ is unknown. In the cases $ d= 2,\ 4,\ 8,\ 24 $, the conjectured value is
\begin{equation}\label{conj_val}
 \c=|\Lambda_d|^{s/d}\zeta_{\Lambda_d}(s),\ \quad s>d, 
\end{equation}
where $ \Lambda_d $ denotes, respectively, the hexagonal, $ D_4,\ E_8 $ and Leech lattices; $ |\Lambda_d| $ is the volume of its fundamental cell; and $ \zeta_{\Lambda_d} $ is the corresponding Epstein  zeta-function; see \cite[Conjecture 2]{BrHaSa12}. As shown in \cite[Proposition 1]{BrHaSa12}, the  conjectured values \eqref{conj_val} serve as upper bounds for their respective $ \c $.

One way of generalizing Theorem \ref{poppy seed} so that it yields non-uniform limiting distributions was studied in \cite{borodachov2008asymptotics}, where the Riesz potential is multiplied by a weight satisfying semicontinuity conditions. More precisely, one  minimizes the energy
\[  E^w_s(\omega_N):= \sum_{x\neq y\in\omega_N} \frac{w(x,y)}{|x-y|^s},  \]
for a non-negative weight function $ w $ on $ A\times A $. Our present goal is to develop an alternate approach by introducing an external field equipped with a suitable scaling factor that depends on the number of points $ N $. 

With regards to practical implementation, it is worth mentioning that by using a localized weight $ w(\cdot, \cdot):=w_N(\cdot, \cdot) $ that depends on the number of points, one can lower the computational complexity of $ E^w_s(\omega_N) $. This approach is investigated in \cite{borodachov2014low}. On the other hand, a number of papers are dedicated to producing well-distributed discrete configurations by drawing them from a suitable random process with, perhaps, further local optimization, see for example \cite{alishahi2015spherical}, \cite{beltran2015energy}, \cite{mak2016support}. 
The possibility of introducing a multiplicative weight together with an external field, as well as decreasing the complexity of the method described below will be the subject of a future work.


The outline of the paper is as follows. In the remaining part of this section we introduce some essential notation. Section \ref{s_main} contains an extension of the Poppy-seed theorem to the case when an external field is present; it also includes results on separation and covering of minimizing configurations. We discuss numerical examples in Section \ref{s_examples}. Finally,  Section \ref{s_proofs} contains proofs of the results stated in Section \ref{s_main}.


\subsection{Notation} We consider configurations of points restricted to a compact set $A \subset \R^p $,  such that $\mathcal{H}_d(A)>0$,  $d\leq p $. The \textit{external field} $q:A\to (-\infty,\infty] $ is assumed to be lower semi-continuous and finite on a subset of $ A $ of positive $ \H $-measure. We write $ \mathring{M} $ for the interior of a set $ M \subset \R^p $, and $ \overline{M} $ for its closure. For a real number $ r $, let $ (r)_+:= \max (0,\,r) $. The closed ball in $ \R^p $ of radius $ r $ centered at the point $ x $  is denoted by  $ B(x,r) $.  Notation $ L^1(A,\lambda) $ stands for the class of functions integrable on the set $ A $ with respect to measure $ \lambda $.

The \textit{minimal $ (s,d,q) $-energy} of the set $ A $ over all $ N $-point subsets of $ A $ is given by
\begin{equation}\label{min_energy_definition}
\E (A,N):=\inf \{ E^q_{s,d}(\omega_N): \omega_N\subset A,\ \# \omega_N=N \},
\end{equation}
where $ \#S $ denotes the cardinality of a set $ S $. Since $ q $ is lower semi-continuous and $ A $ is compact, there exists a configuration of $ N $ charges $ \hat{\omega}_N $ for which the infimum in \eqref{min_energy_definition} is attained; i.e., 
\[ \e(\hat{\omega}_N)=\E(A,N). \]
Such a configuration $ \hat{\omega}_N $ will be called an \textit{$ N $-point} $ (s,d,q) $\textit{-energy minimizer on $ A $}. 

\section{Main results}\label{s_main}

\subsection{A Poppy-seed theorem for $ \boldsymbol{(s,d,q)} $-energy} The following two results extend Theorem \ref{poppy seed} to $ (s,d,q) $-energy. 

\begin{theorem}\label{thm_weak_lim}
Assume $0<d\le p$, and $s\geq d$. Let $A\subset \R^{p} $ be a $ d $-rectifiable compact set, $ \mathcal{H}_d(A)>0 $, and in the case $ s=d $ require additionally that $ A $  be a subset of a $ d $-dimensional $ C^1 $-manifold. Further assume that $q$ is a lower semi-continuous function on $ A $ and finite on a set of positive $ \H $-measure. 
Define $L_1$ and $ \mu_A^q $ to be the positive constant and  probability measure  determined, respectively, by
\begin{equation}\label{L}
\int \left(\frac{L_1-q(x)}{\c(1+s/d)}\right)^{d/s}_+\d\H(x)=1, \qquad
\d\mu_A^q : = \left(\frac{L_1-q(\cdot)}{\c(1+s/d)}\right)^{d/s}_+\d\H,
\end{equation}
where $ \c $ for $ s\geq d  $ is the same as in Theorem \ref{poppy seed}.
Then  
\begin{equation}\label{S_lim}
 \lim\limits_{N\to\infty}  \frac{ \E(A,N) }{\tau_{s,d}(N)}= S(q,A):=  \int\limits \frac{L_1+sq(x)/d}{1+s/d}  \, \d\mu_A^q(x).
\end{equation} 
Furthermore, if $ \{ {\omega}_N\}_{N\geq2} $ is any sequence of asymptotically $ (s,d,q) $-energy minimizing configurations on $A$; that is, 
\begin{equation}\label{asymp_en_min}
	\lim_{N\to\infty} \frac{ \e(\omega_N) }{\T(N)} = S(q,A), 
	\end{equation}
then
\begin{equation}\label{weak_lim}
\frac{1}{N}\sum_{x\in {\omega}_N} \delta_{ x}\stackrel{*}{\longrightarrow}  \d\mu_A^q \quad \text{ as } N\to\infty.
\end{equation}
\end{theorem}

\begin{remark}
As with Theorem \ref{poppy seed}, this result holds on the (possibly) larger class of sets $ A $ satisfying $ \mathcal{H}_d(A)=\mathcal{M}_d(A) $, where $ \mathcal{M}_d $ is the $ d $-dimensional Minkowski content.
\end{remark}

As an application of Theorem \ref{thm_weak_lim}, we deduce a method for  constructing a sequence of $ (s,d,q) $-energy minimizing collections $ \hat{\omega}_N $ such that their normalized counting measures weak$ ^* $ converge to a given distribution.
\begin{theorem}\label{thm_recovery_bounded}
Let the assumptions of Theorem \ref{thm_weak_lim} on the set $ A $ and numbers $ s,d,p $ hold. Assume further $\rho:A\to[0,\infty)$ is an upper semi-continuous function, such that $ \rho\d\H  $ is a probability measure. Then  the lower semi-continuous function $q:A\to\left(-\infty,0\right]$ given by
\begin{equation}\label{q_definition}
q(x):= -M_{s,d}\rho(x)^{s/d},\quad \text{where } M_{s,d}:=\c(1+s/d),
\end{equation}
is such that any sequence  $\{\hat{\omega}_N\}_{N\geq2}$ of $(s,d,q)$-energy minimizers converges weak$ ^* $  to  $\rho \d\H$:
\begin{equation}\label{density_bounded}
\frac{1}{N}\sum_{x\in\hat{\omega}_N} \delta_x\stackrel{*}{\longrightarrow}  \rho \d\H,\quad N\to\infty. 
\end{equation}
In particular, for $ s=d $ equation \eqref{density_bounded} holds with \eqref{q_definition} taking the form 
\begin{equation}\label{q_for_s=d}
 q(x):= - \frac{2\pi^{d/2}\rho(x) }{\Gamma\left(d/2+1 \right) }. 
\end{equation}
\end{theorem}
\begin{remark}
	The reader will no doubt observe that except for the case $d=1$ which is covered in \eqref{cs1}, the usefulness of the last theorem is limited by the lack of knowledge of the value of $ \c $ when $d\geq 2 $. Fortunately, the limit  distribution in equation \eqref{density_bounded} is stable under perturbations of the constant $ M_{s,d} $: small error in the value of $\c $ used in \eqref{q_definition} only leads to small errors in the resulting weak$ ^* $ limit of minimizers. We quantify this statement in Proposition \ref{prop_stability} below.
	
	Another possible way of overcoming this difficulty is modifying the problem of minimizing \eqref{s_energy} so that the charges are restricted to an unbounded set $ A $. It will be addressed in a future work.
\end{remark}

\begin{proposition}\label{prop_stability}
	Assume that in Theorem \ref{thm_recovery_bounded} one uses an approximate value of $ \c $ satisfying
	\[ \c'=  (1+\Delta) \c  \]
	with a fixed $ \Delta $. Let also $ \rho(x)\geq \delta >0 $ for all $ x\in A $ and write $ q'(x) $ for the external field defined with $ \c' $ instead of $ \c $ in \eqref{q_definition}. Then for $ \Delta < {M_{s,d}}/  \left({1+(\|\rho\|_\infty \delta^{-1} )^{s/d} }\right) $, the weak$ ^* $ limit of the $ (s,d,q') $-energy minimizers has density $ \rho' = \d\mu_A^{q'} / \d\H  $ satisfying
	\begin{equation}\label{eq-stability}
	|\rho'(x) - \rho(x)|\leq  \Delta  \frac{ d(1 + \|\rho\|_\infty^{s/d} /\rho(x)^{s/d} ) }{sM_{s,d}}  + o(\Delta), \qquad \Delta \to 0.
	\end{equation}
\end{proposition}

\begin{example}
	Consider the problem of minimization of $ (4,1,q) $-energy on the interval $ [0,2] $, where 
	\[ q= (x-1)^2+\frac{1}{2}. \]
	Formula \eqref{cs1} gives the exact value of $ C_{4,1} $, which enables us to plot the density of $ \mu_{[0,2]}^q $ on $ [0,2] $. For comparison, we also plot the densities of asymptotic distributions obtained for non-exact values of $ C_{4,1} $ by taking $ \Delta = \pm 0.3, \pm 0.25, \pm 0.2, \pm 0.1, \pm 0.05, \pm 0.03 $ in Proposition \ref{prop_stability}.
	
	\begin{figure}[ht!]
		\centering
		\begin{subfigure}{.7\linewidth}
			\includegraphics[scale=0.4]{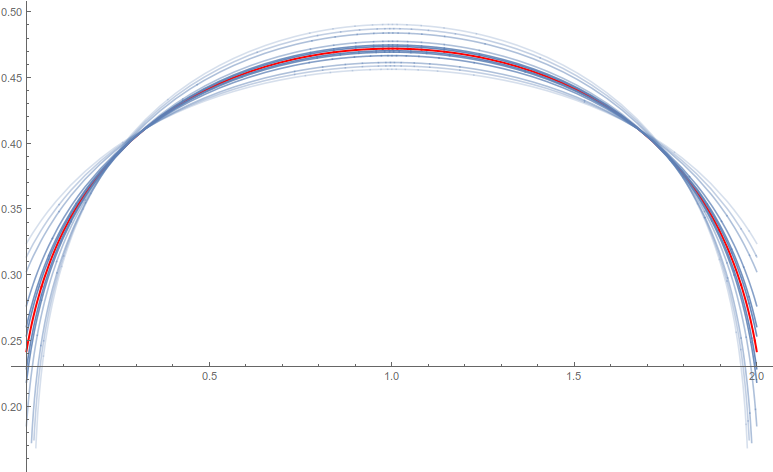}
		\end{subfigure}
		\caption{The exact graph of density $ \d\mu_{[0,2]}^q $ is pictured in red, graphs for perturbed values of $ C_{4,1} $ are in blue.}
		\label{fig:stability}
	\end{figure}
	
\end{example}


\subsection{Separation and covering properties of minimal configurations}\label{section_separation}
Let
\[ \delta(\omega_N):=\min_{\substack{x\neq y,\\ x,y\in \omega_N}}|x-y|, \]
be the \textit{separation distance} of configuration $ \omega_N $. We write $ A(u):= \{ x\in A : q(x)\leq u \} $ for a $ u \in \R $.
\begin{theorem}\label{separation}
Let $0<d\leq p$ and $s\geq d$. Let $ A\subset \R^{p}$ be compact with $ \mathcal{H}_d(A)>0 $, and  let $ q  $ be a nonnegative lower semi-continuous function on $ A $. Then there exists a constant $ C(A,s,d,q) $ such that for each $ N $-point  $ (s,d,q) $-energy minimizer $ \hat{\omega}_N\subset A $  
\[ \delta(\hat{\omega}_N )\geq C(A,s,d,q) \begin{cases}
N^{-1/d} & s >d, \\
(N\log N )^{-1/d}  & s=d,
\end{cases}   \qquad  N\geq 2. \]
\end{theorem}
To prove Theorem \ref{separation}, we will need the following lemma which is also of independent interest. For a sequence of configurations $ \{\omega_N\}_{N\geq2} $ we consider the quantity
\begin{equation}\label{variable_point}
U(x,\omega_N):= \sum_{\substack{y\in {\omega}_N:\\ y\neq x  }} |x-y|^{-s}+ q(x)\tau_{s,d}(N)/N.
\end{equation}
\begin{lemma}\label{separation_lemma}
	Let the assumptions of Theorem \ref{separation} be satisfied. Then there exists a constant $C(A,s,d,q)$ such that for  every $(s,d,q)$-energy minimizing configuration $\hat{\omega}_N, \ N \geq 2, $ and each $ x\in\hat{\omega}_N $ there holds
	\begin{equation} \label{separation_variable_x}
	U(x,\hat{\omega}_N)\leq C(A,s,d,q) \begin{cases}
	N^{s/d} & s >d; \\
	N\log N    & s=d, 
	\end{cases} \qquad  N\geq 2.
	\end{equation}
\end{lemma}
\begin{corollary}\label{restrict-minimizers}
	Let the assumptions of Theorem \ref{separation} hold. Then there exists a constant $ C = C(A,s,d,q) $ such that for all $ N\geq 2 $, the minimizers $ \hat{\omega}_N $ are contained in the set $ A(C) $.
\end{corollary}
Due to this corollary, the sets $ \{\hat{\omega}_N\}_{N\geq2} $ for the problem of minimizing the $ (s,d,q) $-energy on the whole space $ \R^d $ are restricted to a compact set, provided that for some compact $ A  $ and a large enough cube $ Q: = [-R,R]^d $ with $ A\subset Q $, the value  $ C $ in \eqref{separation_variable_x} is such that $ q(x) > C $ for any $ x $ not in $ Q $. Such a problem is then equivalent to energy minimization on $ Q $ only.

To prove the covering property of $ (s,d,q) $-energy minimizers, we will need the notion of Ahlfors regularity \cite[Definition 1.13]{david1993analysis}. 
A set $ A\subset\R^{p} $ with $ \mathcal{H}_d(A)>0 $ is called $ d $\textit{-regular with respect to} $ \lambda $ if  there  are positive constants $ c_0,\ C_0 $ and a positive locally finite Borel measure $ \lambda $, such that 
\begin{equation}\label{regular_WRT_measure}
c_0 R^d\leq \lambda\big( B(x,R)\cap A \big) \leq C_0 R^d 
\end{equation}
for all $ x\in A $ and $ 0<R\leq \diam A $. In the case $ \lambda = \h $, the set $ A $ is called \textit{Ahlfors regular with dimension $d$}.

For an $ x\in A $ and an $ N $-point collection $ \omega_N  $ define 
\[ \dist(x, {\omega}_N) :=\min_{y \in {\omega}_N} |y-x |, \]
the \textit{covering radius at} $ x $ with respect to $ {\omega}_N $.

\begin{theorem}\label{covering}
Let $0<d\leq p$ and $s > d$. Assume $A\subset \R^{p} $  is compact, $d$-rectifiable and Ahlfors regular with dimension $ d $. Assume also $q\geq 0 $ is a continuous function. Let $ x\in  A(L_1-h) $ for some $ h>0 $, where $ L_1 $ is defined in Theorem \ref{thm_weak_lim}. Then  for each sequence of $ (s,d,q) $-energy minimizers $ \{\hat{\omega}_N\}_{N\geq2} $, there exists a constant $  C(A, h,s,d,q) $ such that
\[\dist(x,\hat{\omega}_N)\leq C(A,  h,s,d,q)  N^{-1/d}, \qquad N\geq 2.\]
\end{theorem}
A sequence of configurations $ \{\omega_N \}_{N\geq1} $ is said to be \textit{quasi-uniform in }$ M\subset A $ if the ratio
\begin{equation}\label{mesh_ratio}
 \gamma(x;\omega_N,A):=\dist(x,{\omega}_N)/\delta(\omega_N) 
\end{equation}
is bounded uniformly for all $ x\in M $ and $ N\geq 1 $. From Theorems \ref{separation} and \ref{covering} we have the following result.
\begin{corollary}\label{thm_mesh_ratio}
Let $ s>d $. Assume $ A\subset\R^{p} $ is compact, $ d $-rectifiable and  Ahlfors regular with dimension $ d $. Suppose also that $ q: A\to \R $ is a continuous function. Then for any sequence of $ (s,d,q) $-energy minimizers $ \{\hat{\omega}_N\}_{N\geq2} $ on $ A $,
sequence of subsets  $ \{\hat{\omega}_N\cap A(L_1-h)\}_{N\geq2} $ is quasi-uniform in $ A(L_1-h)$ for any $ h>0 $. That is, for some constant $ C= C(A,  h,s,d,q) $ there holds:
\[ \gamma(x;\omega_N,A(L_1-h)) \leq C(A,  h,s,d,q), \qquad x\in A(L_1-h), \ N \geq 2. \]

\end{corollary}

\section{Examples and numerics}\label{s_examples}

All the results of this section were obtained by using default Mathematica routines (\mbox{FindMinimum}) to minimize the energy functional, starting with a randomly generated collection of point charges. We will write $ L_1(q,A) $ to show explicitly the set on which we are solving the minimization problem and the external field acting on it.

In this section $ \mathbf{e}_z:=(0,\ 0,\ 1)^T $ is the basis vector.

\begin{example}\label{mult}

Consider the problem of minimizing \eqref{s_energy} with $ s=2 $ and an external field
\[ q_\text{a}(x)= \cos (3\arccos \langle x,\mathbf{e}_z \rangle )^{16}   \]
on the unit sphere $ \mathbb{S}^2\subset\R^3 $.  According to \eqref{c_dd}, $ C_{2,2} =\pi $. Equation \eqref{L} for $ L_1(q_\text{a},\mathbb{S}^2) $  in this case is 
\begin{equation}\label{L_example}
 \int_{\mathbb{S}^2} \left(\frac{L-q_\text{a}(x) }{2\pi}\right)_+ \d \mathcal{H}_2(x) =1, 
\end{equation}
solving it for $ L $ gives $ L_1(q_\text{a},\mathbb{S}^2) \approx 0.65448 $. Figure \ref{fig:mult_profile} is the  graph of $ q_\text{a} $ depending on $ \langle x,\mathbf{e}_z \rangle $.

\begin{figure}[h!]
	    \centering
    \begin{subfigure}{.45\linewidth}
      \includegraphics[scale=0.4]{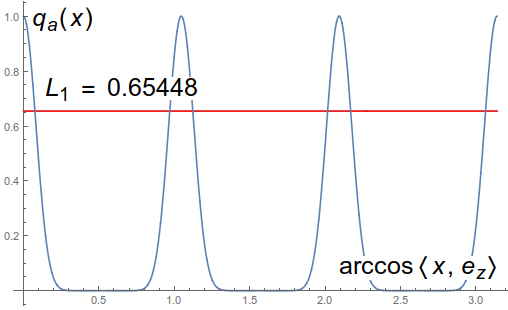}
    \end{subfigure}
    \hskip2em
    \begin{subfigure}{.45\linewidth}
    	\includegraphics[scale=0.4]{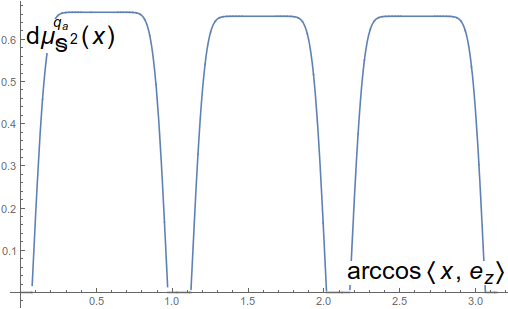}
    \end{subfigure}
     \caption{ Left: graph of $ q_\text{a}(x) $, right: $ \d\mu_{\mathbb{S}^2}^{q_\text{a}}(x) $ from Example \ref{mult}. The horizontal axis is $\arccos (\langle x,\mathbf{e}_z \rangle) $. }
    \label{fig:mult_profile}
\end{figure}
Density of $ \mu_{\mathbb{S}^2}^{q_\text{a}} $ for this external field is
\[ \d\mu_{\mathbb{S}^2}^{q_\text{a}} (x) = \left(\frac{   L_1(q_\text{a},\mathbb{S}^2)  - \cos (3\arccos \langle x,\mathbf{e}_z \rangle )^{16} }{2\pi}\right)_+ \d\H. \]
Using the numeric method described above, we obtain an approximate minimizer $ \omega_\text{a} $ pictured in Figure  \ref{fig:mult}.

\begin{figure}[h!]
    \centering
\begin{subfigure}{.45\linewidth}
        \includegraphics[scale=0.35]{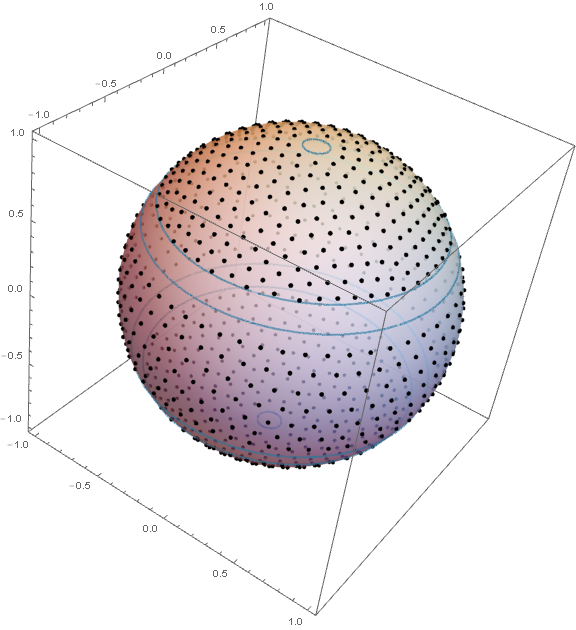}
    \end{subfigure}
    \hskip2em
    \begin{subfigure}{.45\linewidth}
        \includegraphics[scale=0.35]{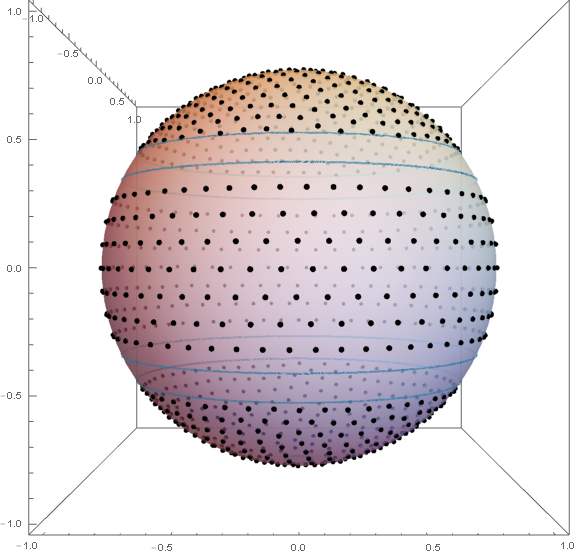}
    \end{subfigure}
     \caption{ Two views of an approximate $ 1000 $-point $ (2,2,q_\text{a})$-energy minimizer $ \omega_\text{a} $ from Example \ref{mult}. The latitudinal circles denote the boundaries of $ \supp  \mu_{\mathbb{S}^2}^{q_\text{a}} $; i.e., $ \{x: q_\text{a}(x)=L_1(q_\text{a},\mathbb{S}^2) \} $. } 
    \label{fig:mult}
\end{figure}

Evaluating separation distance for  $ \omega_\text{a} $ gives $ \delta( \omega_\text{a} )\approx 0.0813 $. Covering radius for the middle strip is $ \eta_{\text{mid}}\approx 0.0829 $, and for the other two  $ \eta_{\text{polar}}\approx 0.0727$, whence mesh ratio is $ \gamma_{\text{mid}}\approx 1.02 $ and $ \gamma_{\text{polar}}\approx 0.8942 $ respectively.
\end{example} 
\begin{example}\label{caps}
	
	\begin{figure}[h!]
		\centering
		\begin{subfigure}{.45\linewidth}
			\includegraphics[scale=0.45]{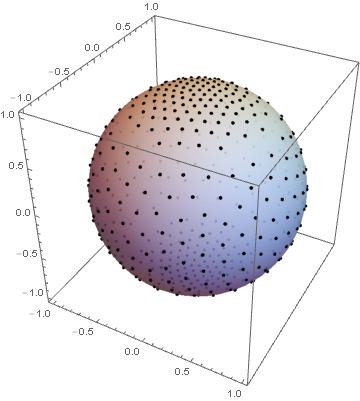}
		\end{subfigure}
		\hskip2em
		\begin{subfigure}{.45\linewidth}
			\includegraphics[scale=0.45]{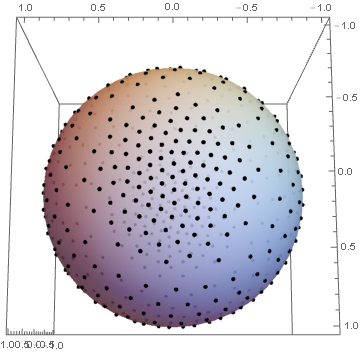}
		\end{subfigure}
		\caption{Two views of an approximation of $ 500 $-point $ (2,2,q_2)$-energy minimizer $ \omega_\text{b} $ from Example \ref{caps}. } 
		\label{fig:caps}
	\end{figure}
	
Again, let $ A=\mathbb{S}^2\subset\R^3, \ s=d=2 $. Let us construct a sequence of discrete collections $ \{\hat{\omega}_N\}_{N\geq2} $ weak$ ^* $ converging to the probability distribution with density proportional to 
\begin{equation}\label{caps_density}
\rho_\text{b}(x)= \begin{cases}
10\cos(4\phi)+11, & 0\leq\phi<\pi/4, \\
1, & \pi /4\leq \phi <3\pi/4,\\
10\cos(4\phi)+11, & 3\pi/4\leq\phi,
\end{cases}
\end{equation}
where  $ \phi=\arccos (\langle x,\mathbf{e}_z \rangle ) $. The external field with such a sequence of minimizers is provided by Theorem \ref{thm_recovery_bounded}. Writing $ \rho $ for the normalization of \eqref{caps_density},   $ \rho(x):= \rho_\text{b}(x)/\int_{\mathbb{S}^2} |\rho_\text{b}|\d\mathcal{H}_2\approx \rho_\text{b}(x)/ 5.581722 $, equation  \eqref{q_definition} gives the following external field:
\[ q_\text{b}(x):=-2\pi\rho, \]
where we used again that $ C_{2,2}=\pi $.

An approximate discrete minimizer of this $ (2,2,q_\text{b}) $-energy is shown in the Figure \ref{fig:caps}.
Note how higher density of $ \mu_{\mathbb{S}^2}^{q_\text{b}} $  (equivalently, larger values of $ \rho $) causes charges to concentrate near the poles. Evaluating separation distance for the pictured configuration $ \omega_\text{b} $ gives $ \delta(\omega_\text{b})\approx 0.0777 $, covering radius $ \eta_{\text{b}} \approx 0.1681 $. The mesh ratio of $ \omega_\text{b} $ is therefore: $ \gamma(\omega_\text{b},\mathbb{S} ^2)\approx 2.163 $. 
\end{example}

\begin{example}\label{torus}
	\begin{figure}[h!]
		\centering
		\begin{subfigure}{.45\linewidth}
			\includegraphics[scale=0.45]{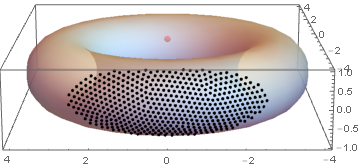}
		\end{subfigure}
		\begin{subfigure}{.45\linewidth}
			\includegraphics[scale=0.45]{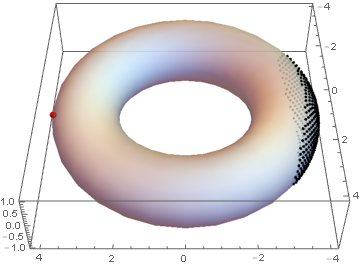}
		\end{subfigure}
		\caption{An approximation of $ 500 $-point $ (8,2,q_{\text{c}})$-energy minimizer $ \omega_c $ from Example \ref{torus}. The red dot marks position $ (4,\ 0,\ 0)^T $, where the repelling external field $ q_\text{c} $ is centered. } 
		\label{fig:torus}
	\end{figure}
	In this example the underlying set $ A $ is a $ 2 $-dimensional torus with inner radius $ r_i = 2 $, outer radius $ r_o = 4 $, centered at the origin. In particular, the point $ (4,0,0) $ lies on the outer side of its surface. Consider the problem of minimizing $( 8,2,q_{\text{c}} )$-energy with the external field 
	\[  q_{\text{c}}(x):=   \|x- (4,\ 0,\ 0)^T\|^{-2}   . \]  A resulting approximation of $ 500 $-point minimizer $ \omega_\text{c} $ is shown in Figure \ref{fig:torus}. Separation distance for this collection is $ \delta(\omega_\text{c}) \approx 0.125339 $.
\end{example}

\begin{example}\label{repelling}
	\begin{figure}[h!]
		\centering
		\begin{subfigure}{.45\linewidth}
			\includegraphics[scale=0.45]{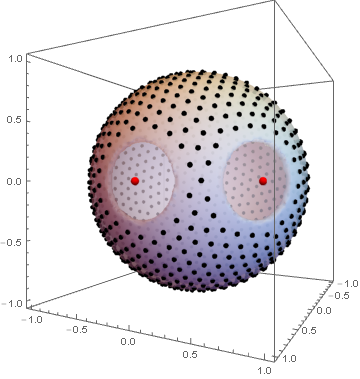}
		\end{subfigure}
		\begin{subfigure}{.45\linewidth}
			\includegraphics[scale=0.45]{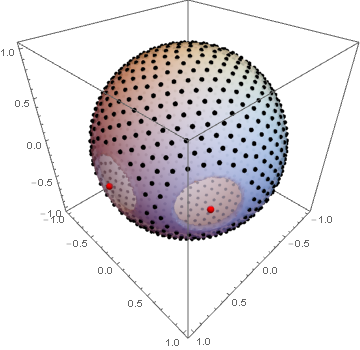}
		\end{subfigure}
		
		\caption{Two views of an approximation of $ 1000 $-point energy minimizer from Example \ref{repelling}. The support of $ \mu_{\mathbb{S}^2}^{q_\text{d}} $ is highlighted as are the positions of the fixed ``repelling charges"  that create the external field  $ q_\text{d} $. } 
		\label{fig:repelling}
	\end{figure}
	  Let us now consider an example of repelling field on the sphere $ \mathbb{S}^2\subset\R^3 $. Namely, we will minimize the $ (4,2,q_\text{d}) $-energy, where 
	\[ q_\text{d}:= 10^{-3}  \left(\left\| x- (1,0,0)^T\right\| ^{ -4 }+\left\| x- ({0.5691,\ 0.8223,\ 0})^T\right\| ^{ -4 }\right). \]
	The second repelling charge is a randomly selected point in the first quadrant of $ Oxy $ plane; factor $ 10^{-3} $ is used merely for convenience purposes.
	
	An approximate 1000-point minimizer $ \omega_\text{d} $ is shown in Figure \ref{fig:repelling}. The shaded region marks the support of $ \mu_{\mathbb{S}^2}^{q_\text{d}} $, obtained using formulas \eqref{L} with $ C_{4,2}\approx 5.7834 $ computed by the formula for its conjectured value \eqref{conj_val}. In other words, the shaded set is $ \{x: q_\text{d}(x)\leq L_1(q_\text{d},\mathbb{S}^2) \}\approx \{x: q_\text{d}(x)\leq 0.127 \} $ (thus the complement of the support in the sphere consists of two circular-like regions). The separation distance of the pictured configuration is $ \delta(\omega_\text{d}) \approx 0.1015 $.
\end{example}

\begin{example}\label{1d} 
	\begin{figure}[h!]
		\centering
		\begin{subfigure}{.45\linewidth}
			\includegraphics[scale=0.35]{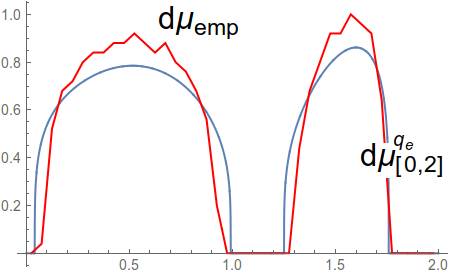}
		\end{subfigure}
		\hskip2em
		\begin{subfigure}{.45\linewidth}
			\includegraphics[scale=0.35]{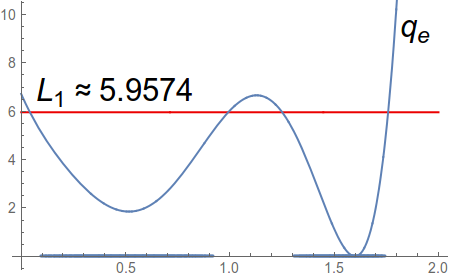}
		\end{subfigure}
		\caption{Left, empirical density $ \d \mu_\text{emp} $ of $ \omega_\text{e} $, an approximate $ 500 $-point $ (4,1, q_{\text{e}}(x) ) $-energy minimizer from Example \ref{1d} (red) overlaid with the graph of $ \d\mu_{[0,2]}^{q_\text{e}} $ (blue); right,  $ \omega_\text{e} $ and the graph of external field $ q_\text{e} $. } 
		\label{fig:1d}
	\end{figure}
		Finally, consider a $ 1 $-dimensional example. We will minimize the $ (4,1,q_{\text{e}}) $-energy on the interval $ [0,2] $, where 
		\[ q_{\text{e}}(x):= (x-1.6)^4+40 (x-0.2)^4 (x-1.6)^2. \]
		Substituting values of $ s,d, \c $ (the latter using formula \eqref{cs1}) into \eqref{L} gives that the weak$ ^* $ limit of minimizers is the measure with density
		\[ \d\mu_{[0,2]}^{q_\text{e}}  \approx \left(\frac{5.9574-q_\text{e}(\cdot)}{10.8232} \right)^{1/4}_+\d\mathcal{H}_1. \]
		Figure \ref{fig:1d} shows graphs of empirical density computed for $ \omega_\text{e} $ and $ \d\mu_{[0,2]}^{q_\text{e}} $, as well as the graph of $ q_\text{e} $. Separation distance of the pictured configuration is $ \delta(\omega_\text{e})\approx 0.0051 $.
\end{example}

\section{Proofs}\label{s_proofs}
	For $s\geq d$, we denote by $\cl$ the collection of all compact $d$-rectifiable sets  $A\subset\R^{p} $  with $\mathcal{H}_d(A)>0 $ and, in the case $ s=d $,    additionally require $A$ to be a subset of a $ d $-dimensional $C^1$-manifold in $\R^{p}$.  For $A\in \cl$ and a suitable external field $q$, the values of $ L_1 $ and $ S(q,A) $ are defined by \eqref{L} and \eqref{S_lim}, respectively. For real sequences $ \{\zeta_N\}_1^\infty $, $ \{\xi_N\}_1^\infty $ we shall use the notation  $  \zeta_N\sim \xi_N ,\ N\to\infty $ to mean $ \zeta_N/\xi_N \to1,\ N \to \infty $. 
	
	Observe that in the formula \eqref{s_energy} the scaling factor $ \T(N) $ depends on $ N $, the number of elements in $ \omega_N $. We will occasionally need to evaluate the $ (s,d,q) $-energy of a discrete $ \omega \subset A $ with $ \#\omega \neq N $ and the scaling factor $ \T(N) $, that is, the value of
	\begin{equation}\label{esd-N}
	 \e (\omega,N):= E^{0}_{s,d}(\omega) + \frac{\T(N)}{N }\,\sum_{x\in\omega} q(x),\quad s\geq d.
	\end{equation}
	Throughout this section $ \N $ stands for the set of positive integers.
For $s\geq d $ we also define 
\begin{equation*} 
\g(A) := \liminf_{N\to \infty} \frac{\E (A,N)}{\T(N) },\qquad
\G (A) := \limsup_{N\to \infty} \frac{\E (A,N)}{\T(N)},
\end{equation*}
and if the limits coincide, the common value is denoted by
\[g^q_{s,d} (A)=\g(A)=\G(A). \]
\begin{remark}\label{finite_limits}
Note that for $s\geq d$ both the lower and upper asymptotic limits are finite if $ q $ is finite-valued on a set of positive measure. Indeed, then there exists an $ \L  $ such that $ \H(\{ x\in A : q(x)\leq L_* \} ) >0 $, so the fact that $ \g(A) $ and $ \G(A) $ are finite follows from Theorem \ref{poppy seed} and a simple observation: for any two functions $ q_1,\ q_2 $ satisfying the hypotheses of Theorem \ref{thm_weak_lim}, the inequality $ q_1(x)\leq q_2(x), \ \forall x\in A, $ implies $ \mathcal{E}^{q_1}_{s,d}(A,N)\leq \mathcal{E}^{q_2}_{s,d}(A,N)$. It suffices to put $ q_1:=q,\  q_2(x)\equiv \L $ and restrict  the minimization problem to $\{ x\in A : q(x)\leq L_* \}$.
\end{remark}
\begin{remark}\label{L_is_finite}
If $ q $ is finite valued on a set of positive measure, then the constant $L_1$ in \eqref{L} is finite:
$$L_1\leq \c \left(\H(\{x\in A : q(x)<\L\})\right)^{-s/d}(1+s/d) +\L, $$
 where $ \L $ is as in Remark \ref{finite_limits}.
\end{remark}
\begin{remark}
	We will use in many computations that if a sequence  $ \{a_N\}_{N\in\n} $ with $ a_N>0 $ satisfies $ \lim_{\n\ni N\to\infty}  a_N/{N} = \alpha \geq 0 $, then
	\begin{equation}\label{T-limit}
	 \slim \frac{\T\left(a_N\right)}{\T(N)}  = \alpha^{1+s/d}.
	\end{equation}
\end{remark}
\subsection{Proofs of the main theorems} 
We first establish a few lemmas that will be used in the proof of Theorem \ref{thm_weak_lim}.
\begin{lemma}\label{function-lemma}
Let $u,v>0 $ and $q_0,q_1$ be real. Then the function 
\begin{equation}\label{F_t}
F(t):= tq_0+(1-t)q_1+ut^{1+s/d} +v(1-t)^{1+s/d},\quad s\geq d,
\end{equation}
has a unique minimum on $[0,1]$. If there is some  $t_*$ in $(0,1)$ that satisfies
\begin{equation}\label{argmin}
\frac{q_1-q_0}{1+s/d}=ut_*^{s/d}-v(1-t_*)^{s/d},
\end{equation}
then the minimum occurs at $t_*$. Otherwise, the minimum occurs at $ t_*\in \{0,1\} $ such that $ q_{1-t_*} = \min  \{ q_0, q_1 \} $.

\end{lemma}
\begin{proof}[\textbf{Proof.}]
As $F$ is strictly convex on $[0,1]$, it has a unique minimum in this interval. Differentiating yields $F'(t) = q_0-q_1 + (1+s/d) ( ut^{s/d}-v(1-t)^{s/d}  )  $. If there is a $ t_*\in (0,1) $ satisfying \eqref{argmin}, then $ F'(t_*) =0 $ and the minimum of $ F $ occurs  at the value $ t_* $. Otherwise, $ F $ is strictly monotone in $ [0,1] $, and the minimium must occur at an endpoint. In fact, the minimum is at $ t_*\in \{0,1\} $ such that $ q_{1-t_*} = \min  \{ q_0, q_1 \} $.
\end{proof}

\begin{lemma}\label{almost-constant}
	 Let $ \lambda  $ be a finite Radon measure on the set $ A \subset  \R^p $ and $ q: A\to \R  $ be measurable with respect to $ \lambda  $. Then for every $ \epsilon >0 $ and $ \lambda $-a.e. point $ x \in \R^p $ there exists a positive number $ R = R(x,\epsilon) $ such that 
	 \begin{equation}\label{almost-constant-eq}
	 \frac{\lambda\left[\{ z\in A\cap B(x,r) : |q(z)-q(x)| < \epsilon  \} \right] }{\lambda[B(x,r)]} > 1-\epsilon
	 \end{equation}
	 for all $ r < R $.
\end{lemma} 
\begin{proof}[\textbf{Proof.}]
	Consider the following partition of set $ A $:
	\begin{equation}\label{Lebesgue_partition}
	A= \bigcup_{m=1}^\infty \left\{(m-1)\epsilon \leq q(x) < m\epsilon \right\}=:\bigcup_{m=1}^\infty {A}_m.
	\end{equation} 
	
	By the Lebesgue-Besicovitch differentiation theorem (cf. \cite[1.7.1]{evans1991measure} or \cite[Corollary 2.14]{mattila1999geometry})  (or \textit{Lebesgue's density theorem} in this case), for $ \lambda $-a.e. point
	$ x\in A_m,\  m=1,2,\ldots, $
	\begin{equation}\label{density-theorem}
	\lim_{r\downarrow 0} \frac{\lambda[A_m\cap {B}(x,r)]}{\lambda[  {B}(x,r)]}=1. 
	\end{equation}
	Therefore, \eqref{density-theorem} holds for  $ \lambda $-a.e. point $ x\in A $. In particular, fix such a point $ x\in A_m $. Because $ (A_m\cap B(x,r)) \subset \{ z\in B(x,r) : |q(z)-q(x)| < \epsilon   \}  $, equation \eqref{density-theorem} implies \eqref{almost-constant-eq} for small enough $ R $.
\end{proof}
 
\begin{lemma}\label{lower-optimum}
	Let $ A \subset \R^p $ satisfy $ A = \cup_{m=1}^M A_m $, where $ A_m  $ are sets from the class  $ \cl $. Let also the function $ q $ be defined and lower semi-continuous on $ A $. Assume that a sequence of configurations $\omega_n  \subset A$, $n\in\n $, is such that
	\begin{enumerate}
		\item $ \omega_n  = \bigcup_{m=1}^M \omega^{m}_n $ and $ \omega^{m}_n\subset A_m $;
		\item $ \omega_n^{k} \cap \omega_n^l = \emptyset $ if $ k\neq l $;
		\item $\lim_{\n\ni n\to\infty}  \#\omega^{m}_n/n = \alpha_m,\   1\leq m \leq  M.$ \label{assumption}
	\end{enumerate}
	Then 
	\begin{equation}\label{cross-lower}
	\liminf_{\n\ni n\to\infty} \frac{E^q_{s,d}(\omega_n,n)}{\T(n)} \geq  \sum_{m=1}^M \alpha_m^{1+s/d}\frac{ \c}{\h(A_m)^{s/d}} + \sum_{m=1}^M \alpha_m \min_{x\in A_m} q(x).
	\end{equation}
\end{lemma}
\begin{proof}[\textbf{Proof.}]
	Observe that the minima on the right hand side of \eqref{cross-lower} are attained due to the lower semi-continuity of $ q $. For the left hand side of \eqref{cross-lower} there holds
	\begin{align*}
	\liminf_{\n\ni n\to\infty} \frac{E^q_{s,d}(\omega_n,n)}{\T(n)} =   & \ \liminf_{\n\ni n\to\infty} \frac{1}{\T (n)} 
	\left( \sum_{ \substack{x\neq y \\ x,y\in\omega_n} }  	|x-y|^{-s} + \frac{\T(n)}{n }\sum_{x\in\omega_n} q(x)
	\right) \nonumber
	\\ 
	\geq &\ \liminf_{\n\ni n\to\infty} \frac{1}{\T(n)} \sum_{m=1}^M  \sum_{ \substack{x\neq y \\ x,y\in\omega^{m}_n} }  	|x-y|^{-s} +  \liminf_{n\to\infty} \frac{1}{n} \sum_{m=1}^M \sum_{x\in\omega^{m}_n} q(x)   \nonumber
	\\ 
	\geq &\  \liminf_{\n\ni n\to\infty} \sum_{m=1}^M \frac{\T(\#\omega^{m}_n) }{\T(n)}  \frac{E^{0}_{s,d}(\omega^{m}_n) }{\T(\#\omega^{m}_n) } 	  +  \liminf_{n\to\infty}  \sum_{m=1}^M \frac{\#\omega^{m}_n}{n} \min_{x\in A_m} q(x)
	\\
	\geq &\  \sum_{m=1}^M \alpha_m^{1+s/d}  \frac{ \c}{\h(A_m)^{s/d}} + \sum_{m=1}^M \alpha_m \min_{x\in A_m} q(x), 
	\end{align*}
	where for the last inequality we used \eqref{T-limit} and Theorem \ref{poppy seed}.
\end{proof}

\begin{remark}
	Observe that the only assertion about $ \#\omega_n $ we make is \eqref{assumption}.
\end{remark}

\begin{corollary}\label{lower-opt-cor}
	Let the assumptions of Lemma \ref{lower-optimum} be satisfied and suppose $ q_m,\ 1\leq m \leq M $, are numbers such that the closure of  $Z_m:=  \{ x\in A_m : q(x)< q_m \} $ has $ \h$-measure zero. Then
	\begin{equation}\label{cross-lower-corollary}
	\liminf_{\n\ni n\to\infty} \frac{E^q_{s,d}(\omega_n,n)}{\T(n)} \geq  \sum_{m=1}^M \alpha_m^{1+s/d}\frac{ \c}{\h(A_m)^{s/d}} + \sum_{m=1}^M \alpha_m   q_m.
	\end{equation}
	\end{corollary}
\begin{proof}[\textbf{Proof.}]
    Let $ \n'\subset \n $ be such that 
    \[  \liminf_{\n\ni n\to\infty} \frac{E^q_{s,d}(\omega_n,n)}{\T(n)} =  \lim_{\n'\ni n\to\infty} \frac{E^q_{s,d}(\omega_n,n)}{\T(n)}.  \]
    Then
    \[ \begin{aligned}
    \lim_{\n'\ni n\to\infty} \frac{E^q_{s,d}(\omega_n,n)}{\T(n)}\geq & 
     \sum_{m=1}^M \alpha_m^{1+s/d}  \frac{ \c}{\h(A_m)^{s/d}} + \lim_{\n'\ni n\to\infty} \frac{1}{n} \sum_{m=1}^M \sum_{x\in\omega^{m}_n} q(x) \\ 
    \geq  &  \sum_{m=1}^M \alpha_m^{1+s/d}  \frac{ \c}{\h(A_m)^{s/d}} +   \lim_{\n'\ni n\to\infty} \frac{1}{n} \left( \sum_{m=1}^{M} \frac{\#(\omega_n^m\cap (A\setminus Z_m) )}{n}  \right)\\
    = & \sum_{m=1}^M \alpha_m^{1+s/d}\frac{ \c}{\h(A_m)^{s/d}} + \sum_{m=1}^M \alpha_m   q_m.
    \end{aligned}
     \]
\end{proof}

\begin{lemma}\label{abs-continuity} Let the set $ A $ be such that the assumptions of Theorem \ref{thm_weak_lim} hold.
	Assume that a sequence of $ N $-point configurations $ \{\omega_N\}_{N \geq 2}   $ in $ A $ satisfies
	\begin{equation}\label{asymp-finite}
		\limsup_{\substack{N\to \infty \\ N\in\mathcal{N}  }} \frac{ \e(\omega_N) }{\T(N)} < + \infty
	\end{equation}
	and
	\begin{equation}
	\frac{1}{N}\sum_{x\in {\omega}_N} \delta_{ x}\stackrel{*}{\longrightarrow}  \d\mu, \qquad \n\ni N \to \infty, 
	\end{equation}
	for some Borel probability measure $ \mu $ on $ A $. Then $ \mu $ is $ \H $-absolutely continuous.
\end{lemma}
\begin{proof}[\textbf{Proof.}]
	Indeed, otherwise let $ E \subset A $ be a Borel set such that $ \h(E) = 0 $ and $ \mu(E)  >0 $. Since $ \mu $ is inner regular as a Borel measure on a Radon space, \cite[434K(b)]{fremlin2000measure}, without loss of generality $ E $ is closed. For an $ \epsilon >0 $ pick $ r >0 $ such that $ E_r := \{ x \in A:  \dist(x,E)\leq r \} $ satisfies $ \h(E_r) < \epsilon  $; observe that $ E_r $ is closed. By the definition of weak$ ^* $ convergence and Urysohn's lemma, $ \lim_{\mathcal{N}\ni N \to \infty } \frac{1}{N} \# \{ x\in{\omega}_N : x\in E_r \} \geq \mu(E)  $ (consider a positive continuous function equal 1 on $ E $ and supported on $ E_r $). Then according to Theorem \ref{poppy seed} and the limit \eqref{T-limit}, 
	\[
	\begin{aligned}
	\liminf_{\substack{N\to \infty \\ N\in\mathcal{N}  }} \frac{E^0_{s,d} ( {\omega}_N \cap E_r,N) }{\T(N)} & = \liminf_{\substack{N\to \infty \\
			N\in\mathcal{N}  }} \frac{E^0_{s,d} ({\omega}_N \cap E_r,N) }{\T(\# ( {\omega}_N \cap E_r) ) } \frac{ \T(\# ({\omega}_N \cap E_r))}{\T(N)} 
	\\ & \geq \frac{\c }{\h(E_r)^{s/d}} \mu(E)^{1+s/d} \geq \frac{\c  }{\epsilon^{s/d}}\mu(E)^{1+s/d}.
	\end{aligned}
	\]
	As $ \epsilon $  was arbitrary, this contradicts \eqref{asymp-finite}.
	Thus  $ \mu $ must be $ \H $-absolutely continuous. 
\end{proof}
\begin{lemma}\label{local-optimal}
	Let the assumptions of Theorem \ref{thm_weak_lim} be satisfied. Let also the sequence of $ N$-point configurations $ \{ {\omega}_N\}_{N\in\mathcal{N}} $ be such that
	\begin{equation}\label{lim-inf}
	\slim  \frac{ \e( {\omega}_N) }{\T(N)} = \g(A)
	\end{equation}
	and
	\begin{equation}\label{weak-convergence}
	\frac{1}{N}\sum_{x\in {\omega}_N} \delta_{ x}\stackrel{*}{\longrightarrow}  \d {\mu}, \qquad \n \ni N \to \infty.
	\end{equation} 
	Assume that $ \{B_m\}_{m=1}^M, \ M\geq 1, $ is a collection of closed pairwise disjoint balls such that $\H(B_m)>0,\  \H(\partial B_m) = 0 $ and $ 
	{\H\left(\{ z\in B_m :  q(z)  \leq q_m  \}\right) }  \geq (1-\delta)\H(B_m)$,   $ m = 1,\ldots,M,$ for some positive $ \delta <1 $.
	
	 Then 
	 \begin{equation}\label{omega-0}
	\slim  \frac{\e ( {\omega}_N\cap\left(\cup_m B_m\right) ,N)}{\T(N)}  
    \leq  \min
    \left\{ \sum_{m=1}^M 
    \frac{ \c\, \alpha_m^{1+s/d} }{((1-\delta) \H(B_m))^{s/d}} +   q_m \alpha_m   
	 \right\},
	 \end{equation}
     where the minimum is taken over $\alpha_m\geq0$ such that $\sum\alpha_m =\sum\mu(B_m)$.
     
	 In particular, there exists a sequence $ \{\omega_N^0\}_{N\in\n} $ for which \eqref{omega-0} is an equality with $ \omega_N^0 $ in place of $ \omega_N $.
\end{lemma}
\begin{proof}[\textbf{Proof.}] Fix an $ \epsilon >0 $ satisfying $ \epsilon < 1-\delta $. Consider the set $ \{ z\in B_m :  q(z)  \leq q_m  \}, \ m=1,\ldots,M $. By the inner regularity of measure $ \H $, it has a closed subset $ B'_m $ contained in a ball concentric with $ B_i $ of smaller radius, for which $ \H(B'_m) > (1-\delta-\epsilon)\H(B_m) $. Let a sequence of $ N $-point configurations $ \{\omega_N^0\}_{N\in\mathcal{N} }  $ in $ A $ be such that $ \omega_N^0 \cap (A \setminus \cup_m B_m ) =  \omega_N \cap (A \setminus \cup_m B_m ) $, and such that for $ 1\leq m \leq M $, the collection  $  {\omega}_N^0 \cap B_m $ is a minimizer of the $ (s,d,0) $-energy in $ B'_m $ (in particular, is contained in it). 
	
    Equation \eqref{lim-inf} and Lemma \ref{abs-continuity} imply that $ \mu(\partial B_m)=0,\ 1\leq m\leq M $. Hence the weak$ ^* $ convergence in \eqref{weak-convergence} implies $ \lim \#\{ x\in \omega_N : x\in B_m  \}/N \to \mu(B_m) $ when $ \n \ni N \to \infty $, \cite[Theorem 2.1]{billingsley1999convergence}.
    
	We will further assume that the following limits exist $ \alpha_m := \lim_{\mathcal{N} \ni N\to\infty} \#(\omega_N^0 \cap B_m) / N  $, $ 1\leq m \leq M $. The assumptions on $ \#(\omega_N^0 \cap B_m) $ mean that $ \sum_m \alpha_m = \sum_m \mu(B_m) $. Finally, we observe that by the construction of the sets $ B_m' $, there exists a positive $ r $ such that $ \dist( \cup_m B_m', A\setminus \cup_m B_m ) \geq r $.
	Recall that
	\begin{equation}\label{cross-terms-eq}
	\begin{aligned}
	 \e( {\omega}_N^0,N) =   \e & \left( {\omega}_N^0  \cap (\cup_m B_m) ,N \right) 
	 +  \e\left( {\omega}_N^0 \cap (A\setminus\cup_m B_m) ,N\right) \\
	   & 
	 +  \sum_{\substack{x,y\in \omega_N^0,\\ x\in \cup_m B_m,\\ y\in A\setminus \cup_m B_m }} |x-y|^{-s}.
	\end{aligned}
	\end{equation}
	Because $  \omega_N^0 \cap \cup_m B_m = \omega_N^0 \cap \cup_m B'_m  $ and because of the lower bound $ r $ for the distance between  $ \cup_m B'_m $ and $ A\setminus \cup_m B_m $, every term in the last sum is at most $ r^{-s} $.
	
	Using the previous equation and the definition of $ \g(A) $, we have:
	\begin{equation}
	\begin{aligned}
	0  \leq & \slim   \frac{ \e( {\omega}_N^0,N) }{\T(N)} - \g(A) = \slim  \left(\frac{ \e( {\omega}_N^0) }{\T(N)} -  \frac{ \e( {\omega}_N) }{\T(N)}\right)  \\
	\leq & \slim  \frac{\e ( {\omega}_N^0 \cap\left(\cup_m B_m\right) ,N)}{\T(N)}    
	- \slim  \frac{\e ( {\omega}_N\cap\left(\cup_m B_m\right) ,N)} {\T(N)}\\
	& + \slim \frac{N^2}{\T(N)} r^{-s}.        
	\end{aligned}
	\end{equation}
	Since $ \slim {N^2}r^{-s}/{\T(N)} =0 $, there holds
	\begin{equation}\label{local-intermediate}
	\slim  \frac{\e ( {\omega}_N\cap\left(\cup_m B_m\right) ,N)} {\T(N)} \leq 
	\slim  \frac{\e ( {\omega}_N^0 \cap\left(\cup_m B_m\right) ,N)}{\T(N)}.    
	\end{equation}
	From equation \eqref{lim-inf} and Lemma \ref{abs-continuity} follows that $ \mu(\partial B_m)=0,\ 1\leq m\leq M $. Hence the weak$ ^* $ convergence in \eqref{weak-convergence} implies $ \lim \#\{ x\in \omega_N : x\in B_m  \}/N \to \mu(B_m) $ when $ \n \ni N \to \infty $, \cite[Theorem 2.1]{billingsley1999convergence}.
	The construction of the sequence $ \{\omega_N^0\}_{N\in\n}  $ and the limit \eqref{T-limit} therefore imply
	\begin{equation}\label{zero-bound}
	\slim  \frac{\e ( {\omega}_N^0 \cap\left(\cup_m B_m\right) ,N)}{\T(N)} \leq \sum_{m=1}^M \left( \frac{ \c\, \alpha_m^{1+s/d} }{((1-\delta-\epsilon) \H(B_m))^{s/d}} +   q_m \alpha_m   \right).
	\end{equation}
	We have so far only imposed the conditions that $\alpha_1,\ldots, \alpha_M $ are nonnegative and sum to  $\sum_m \mu(B_m) $. Taking $ \epsilon\to0+ $ and minimizing over such $ \alpha_m $ in \eqref{zero-bound} gives \eqref{omega-0}.
\end{proof}

We first prove  Theorem \ref{thm_weak_lim} for the case that $q$ is a suitable simple function. The general case then follows by approximating an arbitrary lower semi-continuous $ q $ with such functions.

\begin{lemma}\label{scalable}
Let $ A \subset \R^p $ be a set from $ \cl $, and $ B_m, \,  1\leq m\leq M $,  be a collection of pairwise disjoint closed balls such that $\H(B_m)>0$ and $\H(A\cap\partial B_m)=0, \  1\leq m\leq M $. Assume also $ q $ is a lower semi-continuous function and for  $ D := \overline{A\setminus \cup_m B_m} $,
\begin{equation}
q(x) = \begin{cases}
q_0, & \H\text{-a.e. }x\in D,\\
q_m, & \H\text{-a.e. }x\in B_m, \quad 1\leq m \leq M,
\end{cases}
\end{equation}
for positive $ q_m,\ 0\leq m\leq M $.

Then equation \eqref{S_lim} holds for the set  $ A $ and function $ q $. 
\end{lemma}
\begin{proof}[\textbf{Proof.}]
	For convenience, let $ A_0:=D $, $ A_m:=B_m,\ 1\leq m\leq M $, in this proof. We will first verify that for some positive $ \{\hat\alpha_m\}_{m=0}^M $ that add up to one,
	\begin{equation}\label{S-lim-intermediate}
	\lim\limits_{N\to\infty}  \frac{ \E(A,N) }{\tau_{s,d}(N)} = 
	 \sum_{m=0}^M \left( \frac{ \c\, \hat{\alpha}_m^{1+s/d} }{ \H(A_m)^{s/d}} +   q_m \hat{\alpha}_m   \right),
	\end{equation}
    where the values of $ \hat{\alpha}_m,\ 0\leq m \leq M $ are such that 
    \begin{equation}\label{hat-alpha-m}
    \begin{aligned}
    ( \hat{\alpha}_0,\ldots, \hat{\alpha}_M ) : = \underset{\substack{\alpha_m\geq0,\\ \sum\alpha_m =1 } }{\arg \min}   \sum_{m=0}^M \left( \frac{ \c\, \alpha_m^{1+s/d} }{ \H(A_m)^{s/d}} +   q_m \alpha_m   \right).
    \end{aligned}
    \end{equation}
	
    \textbf{(\textit{i}).} Due to the weak$ ^* $ compactness of the set $ A $, Corollary \ref{lower-opt-cor} implies 
    
    \begin{equation}
    \begin{aligned}
    \g(A)   \geq \min_{\substack{\alpha_m\geq0,\\ \sum\alpha_m =1 } }  \sum_{m=0}^M \left( \frac{ \c\, \alpha_m^{1+s/d} }{ \H(A_m)^{s/d}} +   q_m \alpha_m   \right).
    \end{aligned}
    \end{equation}
    Let a closed $  D'\subset D $ satisfy $ q(x)\equiv q_0,\ x\in D' $ and   $ \H(D')>(1-\epsilon)\H(D) $. By the same argument as in the proof of Lemma \ref{local-optimal}, for a fixed $ \epsilon > 0 $ we construct a sequence of $ N $-element sets $ \{\omega_N^0\}_{N\in\n} $ such that the subsets $  \omega_N^0 \cap D  $ and $ \omega_N^0 \cap B_m, \ 1\leq m \leq M $ are $ (s,d,0) $-energy minimizing in $ D' $ and $ B_m' $ respectively. 
    Recall that $ B_m' $ are closed subsets of $ A\cap B_m $ satisfying $ \H(B_m')> (1-\epsilon)\H(B_m) $ and $ \dist(\cup_m B_m', D )>0 $. As in Lemma \ref{local-optimal}, we construct $ \{\omega_N^0\}_{N\geq1}  $ so that the following limits exist and are finite $
    \alpha_0 := \lim_{\mathcal{N} \ni N\to\infty} \#(\omega_N^0\cap D') / N $ and $ \alpha_m := \lim_{\mathcal{N} \ni N\to\infty} \#(\omega_N^0 \cap B_m) / N  $, $ 1\leq m \leq M $. Since $ \E(A,N)\leq \e ( {\omega}_N^0,N) $, equation \eqref{cross-terms-eq} implies
    \begin{equation}
    \G(A)\leq  \lim_{N\to\infty}  \frac{\e ( {\omega}_N^0,N)}{\T(N)}  
    =  \sum_{m=0}^M \left( \frac{ \c\, \alpha_m^{1+s/d} }{ ((1-\epsilon)\H(A_m))^{s/d}} +   q_m \alpha_m   \right).
    \end{equation}
    This gives \eqref{S-lim-intermediate} after taking $ \epsilon\to0+ $.
    
    \textbf{(\textit{ii}).} Fix an $ A_m $ with strictly positive $ \hat{\alpha}_m $, say, $ A_0 $ and assume $ q_0 \leq q_m $ for definiteness. Pick any of the remaining sets $ A_k, \ 1\leq k \leq M $ and denote $\beta= \beta(k):= \alpha_0+\alpha_k$. Consider the terms on the right hand side of \eqref{hat-alpha-m} that contain either $ \alpha_0 $ or $ \alpha_k $:
    \begin{equation}
    \bar{F}( {\alpha}_1, {\alpha}_k): =  \sum_{m=0,k} \left( \frac{ \c\,  {\alpha}_m^{1+s/d} }{ \H(A_m)^{s/d}} +   q_m  {\alpha}_m   \right).
    \end{equation}
    Now choose the coefficients of the function $ F(t) $ in Lemma \ref{function-lemma} so that
    \begin{equation}
    F(t) = \  t^{1+s/d}\frac{\c}{\H(A_0)^{s/d}} + 
    t\frac{q_0}{\beta^{s/d}}  +(1-t)^{1+s/d}  \frac{\c}{\H(A_k)^{s/d}}    +(1-t)\frac{q_k}{\beta^{s/d}},
    \end{equation}
    then $ \beta^{1+s/d} F ( \alpha_0 / \beta) =    \bar{F}( {\alpha}_0, {\alpha}_k) $. Because of \eqref{hat-alpha-m}, it must be that $ \hat\alpha_0 / \beta $ is the value $ \hat{t}\in(0,1] $ for which the minimum of $ F(t) $ is attained. According to Lemma \ref{function-lemma}, either $ \hat{t} = 1 $, or
    \begin{equation}\label{steps}
    \frac{q_k-q_0}{\c(1+s/d)}=\left(\frac{\hat{\alpha}_0}{\H(A_0)} \right)^{s/d}- \left(\frac{\hat{\alpha}_k}{\H(A_k)}\right)^{s/d}.  
    \end{equation}
    Equation \eqref{steps} thus applies to any pair of sets in $ A_0,\ldots,A_M $ provided both the corresponding $ \hat{\alpha}_m $'s is positive. Also, if $ \hat{\alpha}_k >0 $, then $ q_k < q_l $ for every $ l $ such that $  \hat{\alpha}_l =0  $.
    To summarize, for some $ L_1 $ there holds 
    \begin{equation}\label{step}
    \left(\frac{\hat{\alpha}_m}{\H(A_m)}\right)^{s/d} = \left(\frac{L_1-q_m}{\c(1+s/d)}\right)_+, \qquad 0\leq m\leq M.
    \end{equation}
    It follows from $ \sum_m \hat{\alpha}_m = 1 $ that the first of equations \eqref{L} is satisfied for this $ L_1 $. 
    
    Finally, we can evaluate the right hand side of \eqref{S-lim-intermediate}:
    \[ \lim\limits_{N\to\infty}  \frac{ \E(A,N) }{\tau_{s,d}(N)} =  \sum_{m=0}^M \left(\hat{\alpha}_m \left(\frac{L_1-q_m}{1+s/d} \right)_+ + q_m \hat{\alpha}_m\right) = \sum_{m=0}^M \hat\alpha_m \frac{L_1+sq_m/d}{1+s/d} , \]
	where in the last equality we used $ \hat\alpha_m = 0 \iff (L_1-q_m)_+=0 $. This implies \eqref{S_lim} because from \eqref{step},  $  \hat{\alpha}_0 = \mu^q_A(D)  $ and $ \hat{\alpha}_m = \mu^q_A(A_m), \ 1\leq m \leq M $, for the $ \mu^q_A $ defined in \eqref{L}.
\end{proof}

\begin{proof}[\textbf{Proof of Theorem \ref{thm_weak_lim}.}]

Note that as $ q $ is lower semi-continuous on the compact set $ A $, it is bounded below there, so we may assume without loss of generality $ q $ is positive.

\textbf{(\textit{i}).} Let $A\in\cl $ and let $0\leq q < C $ for a positive constant $ C $. 
We will further use that the restriction $ \H $ is a Radon measure on $ \R^p $. Namely, \cite[Theorem 7.5]{mattila1999geometry} implies it is locally finite because $ A $ is the Lipschitz image of a compact set in $ \R^d $. It is also Borel regular as a restriction of Hausdorff measure, see \cite[Theorem 4.2]{mattila1999geometry}. Then by \cite[Corollary 1.11]{mattila1999geometry}, $ \H $ is a Radon measure.

Fix from now on a number $ 0<\epsilon<1/4 $.  Apply Lemma \ref{almost-constant} to the measure $ \H $ and function $ q $, denote the set of $ x\in A $ for which there exists an $ R(x,\epsilon) $ as described in the Lemma by $ A' $, and consider covering of $ A' $ by the collection of closed balls $ \{ {B}(x,r): x\in A',\ 0 < r < R(x,\epsilon)\} $. Choose for each $ x\in  A' $ a sequence of radii $ r_{x,k}\to0,\  k\to\infty,\ k\in\N, $ for which
\begin{equation}\label{almost-const-applied}
\frac{\H\left[\{ z\in B(x,r_{x,k}) : |q(z)-q(x)| < \epsilon  \} \right] }{\H[B(x,r_{x,k})]} > 1-\epsilon
\end{equation}
and also $ y\in A\cap {B}(x,r_{x,k})   \implies  q(y) > q(x)-\epsilon$. The latter is possible due to the lower semicontinuity of $ q $. 

Let $ \{ {B}(x,r_{x,k}) \} $ be a Vitali cover of $ A'$, so one can apply the version of Vitali's covering theorem for Radon measures \cite[Theorem 2.8]{mattila1999geometry} to produce a (at most) countable subcollection of pairwise disjoint $ \{B_{j}:= {B}(x_j,r_j) :j\geq1 \} $ for which $ \h\left( A'\; {\setminus} \;  \cup_{j\geq1} B_{j} \right)=0 $.
Using $ \mathcal{H}_d(A)<\infty $, $ \{B_j\}_{j\geq1} $ can be chosen so that $ \h(A\cap\partial B_j)=0,\ j=1,2,\ldots $ (there are uncountable many options for the value of $ r_j $, at most countably many of them positive). Since  $ \h(A\setminus A' )= 0 $, we can fix a  $ J\in \N $ such that $ \h\left(A\setminus\cup_{j=1}^J B_{j} \right)<\epsilon. $
Let $  D:= \overline{A\setminus\cup_{j=1}^{\,J}  \B } $.
As $ \H(\partial B_j )=0, \ 1\leq j \leq J $, there holds $ \H(D)<\epsilon $.

Define the two simple functions $ \Q,\ \q $ to be constant on each $  \B,\ 1\leq j\leq J $:
\begin{align}\label{simple_approximations}
\begin{split}
\Q(x)  &:= \begin{cases} q(x_j)+\epsilon, & x\in \B ,\\
C,  & x\in D\setminus  \bigcup_j \B.
\end{cases} \\ 
\q(x) &:= \begin{cases} \max(0,q(x_j)-\epsilon), & x\in \B\setminus D,\\
0,  & x\in D.
\end{cases} 
\end{split}
\end{align}
Such $ \q,\ \Q $ are lower semi-continuous on $ A $. 
Lemma \ref{scalable} gives equation \eqref{S_lim} applied to  $ \q $ and $ \Q $ on $ A $. Let $ \B':= \{z\in A\cap\B : |q(z)-q(x_j)|<\epsilon \} $. Then,
\begin{equation}\label{double-bound}
 q(x_j)-\epsilon\leq \q(x)\leq q(x) \leq \Q(X)= q(x_j)_\epsilon, \quad x\in \B'.
\end{equation}
In view of \eqref{almost-const-applied} for $ \B' $ and $ \H(D)<\epsilon $, \eqref{double-bound} implies that both $ \q $ and $ \Q $ converge $ \H $-a.e. to $ q $ as $ \epsilon\to0+ $. Since both are bounded by $ C+1 $, the dominated convergence theorem is applicable, and 
  \begin{equation}\label{double-limit}
   \lim_{\epsilon\to0} S(\q,A) = \lim_{\epsilon\to0} S(\Q,A) = S(q,A). 
  \end{equation}
We now estimate $  \lim_{N\to\infty}  { \E(A,N) }/{\tau_{s,d}(N)} $ in terms of $  \lim_{N\to\infty}  { \mathcal{E}^{\q}_{s,d}(A,N) }/{\tau_{s,d}(N)} $ and $  \lim_{N\to\infty}  { \mathcal{E}^{\Q}_{s,d}(A,N) }/{\tau_{s,d}(N)} $. Firstly, by construction $ q(x)\geq \q(x), \ x\in A $, which gives 
\begin{equation}\label{lower_bound_energy}
\E(A,N)\geq \mathcal{E}^{\q}_{s,d} (A,N). 
\end{equation}
On the other hand,
\begin{equation}\label{upper_bound_energy}
\E(A,N) \leq \E(D\cup\bigcup_j \B',N)\leq    \mathcal{E}^{\overline{q}}_{s,d} (D\cup\bigcup_j \B',N)\leq \frac{\T(N)}{(1-\epsilon)^{s/d}} S(\Q,A ),
\end{equation}
where the last inequality follows from \eqref{almost-const-applied} and \eqref{S-lim-intermediate}. 
This proves \eqref{S_lim}. 

	\textbf{(\textit{ii}).} It remains to prove equation \eqref{weak_lim} for a sequence $ \{ {\omega}_N\}_{N\geq2} $ satisfying \eqref{asymp_en_min}.
	 Since the probability measures on $ A $ are weak$ ^* $ compact, one can pick a subsequence $ \{\hat{\omega}_N\}_{N\in\mathcal{N}}\subset\{\omega_N\}_{N\geq2} $ for which the corresponding normalized counting measures have a weak$ ^* $ limit:
	\[  \frac{1}{N} \sum_{x\in\hat{\omega}_N} \delta_x \stackrel{*}{\longrightarrow} \mu, \qquad  \mathcal{N}\ni N\to \infty. \]
	Then $ \mu $ is $ \H $-absolutely continuous by the Lemma \ref{abs-continuity}. 
	Set $ \rho(x):=\frac{\d\mu}{\d \H}(x) $.
	
	 Since the integral $ \int_A \rho \d\H = 1 $ is finite, at $ \H $-a.e. point $ x $ of $ A $ there holds 
	\begin{equation}\label{x-differentiation}
	\lim_{r\to0} \frac{1}{\H( B(x,r))} \int_{B(x,r)} \rho\d\H = \rho(x).
	\end{equation} 
	Fix two distinct points $ x_1,\ x_2 $ for which both \eqref{almost-constant-eq} for measure $ \H $ and \eqref{x-differentiation} hold. Then for an arbitrary fixed $ 0< \epsilon < \min( 1/2, \rho(x_1),\rho(x_2), q(x_1), q(x_2)) $ there exist closed disjoint balls $ B_1:=B(x_1,r_1), \ B_2:=B(x_2,r_2) $ centered around $ x_1, \ x_2 $ such that equations 
	\begin{align}
	& 
	\frac{\H\left[\{  z\in A\cap B_m : |q(z)-q(x_m)| < \epsilon  \} \right] }{\H[B_m]} > 1-\epsilon, \qquad m=1,2,
	\label{lebesgue1} \\	
	&\int_{B_m} \left|\rho(z) -\rho(x_m) \right| \d \H(z) <\epsilon\H(B_m),  \label{lebesgue2} \qquad m=1,2,
	\end{align}
	hold for all closed balls concentric with $ B_m $ of radius at most $ r_m $.
	Without loss of generality we will also require $ \H(\partial B_1 )= \H( \partial B_2 )=0 $ and that $ q(x) \geq q(x_m)-\epsilon $ for all $ x\in B_m,\ m=1,2 $ (which can be assumed by lower semi-continuity). Let $ q_m:= q(x_m), \ m=1,2 $; let also $ q_1\leq q_2 $.
	
	Due to $ \mu $ being absolutely continuous with respect to  $ \H $, the assumption $ \H(\partial B_m )=0,\ m=1,2 $, and the limit \eqref{T-limit}, Lemma \ref{lower-optimum} implies
	\begin{equation}\label{lower-used}
	\begin{aligned}
	\slim   \frac{\e (\hat{\omega}_N\cap(B_1\cup B_2) ,N)}{\T(N)}   \geq  \sum_{m=1,2}  \left( \frac{  \c\, \mu(B_m)^{1+s/d} }{\H(B_m)^{s/d}} 
	+    \mu(B_m) (q_m-\epsilon)  \right).
	\end{aligned}
	\end{equation}
	On the other hand, from Lemma \ref{local-optimal}:
	\begin{equation}\label{optimal-used}
	\begin{aligned}
	\slim   \frac{\e (\hat{\omega}_N\cap(B_1\cup B_2) ,N)}{\T(N)}   \leq 
	\min
	\left\{ 
	\sum_{m=1,2} 
	\frac{ \c\, \alpha_m^{1+s/d} }{((1-\epsilon) \H(B_m))^{s/d}} +    (q_m+\epsilon) \alpha_m   
	\right\} 
	\end{aligned}
	\end{equation}
	with minimum taken over positive $ \alpha_1,\ \alpha_2 $ satisfying $ \alpha_1+\alpha_2= \mu(B_1)+\mu(B_2) $. If we denote

	\begin{equation}\label{starred-alphas}
	\begin{aligned}
	( \hat{\alpha}_1, \hat{\alpha}_2 ) : = \arg\min  \Bigg\{ \sum_{m=1,2}  \bigg(  \frac{ \c\, \alpha_m^{1+s/d} }{((1-\epsilon) \H(B_m))^{s/d}} &  +    (q_m+\epsilon) \alpha_m   \Bigg) \\ 
	&   \boldsymbol{:}  \alpha_1+\alpha_2 = \mu(B_1) +\mu(B_2)   \Bigg\}, 
	\end{aligned}
	\end{equation}
	and argue as in the proof of Lemma \ref{scalable}, we obtain, similarly to \eqref{steps}, that  $ ( \hat{\alpha}_1, \hat{\alpha}_2 ) $ satisfy
	\begin{equation}\label{steps-starred}
	\frac{q_2-q_1}{\c(1+s/d)}= \left(\frac{\hat{\alpha}_1}{(1-\epsilon)\H(B_1)} \right)^{s/d}- \left(\frac{\hat{\alpha}_2}{(1-\epsilon)\H(B_2)}\right)^{s/d}.  
	\end{equation}
	Inequalities  \eqref{lower-used}--\eqref{optimal-used} and the definition of $ ( \hat{\alpha}_1, \hat{\alpha}_2 ) $ give:
	\begin{equation}\label{finally}
	\begin{aligned}
	 & \sum_{m=1,2}  \left( \frac{  \c\, \mu(B_m)^{1+s/d} }{\H(B_m)^{s/d}} 
	 +    \mu(B_m) (q_m-\epsilon)  \right) \\
	 \leq & \sum_{m=1,2} \left( \frac{ \c\,  \hat{\alpha}_m^{1+s/d}}{((1-\epsilon) \H(B_m))^{s/d}}   +    \hat{\alpha}_m (q_m+\epsilon)\right) \\
	 \leq & \sum_{m=1,2} \left( \frac{ \c\, \mu(B_m)^{1+s/d} }{((1-\epsilon) \H(B_m))^{s/d}} +  \mu(B_m)  (q_m+\epsilon)  \right). 
	\end{aligned}
	\end{equation}
	
	Observe that if in the above construction we fix the ball $ B_1 $ and allow $ r_2 \to 0$, the first term on the right hand side of \eqref{steps-starred} is bounded, so the ratio $ \hat{\alpha}_2 / \H(B_2) $ is bounded as well, say, $  \hat{\alpha}_2 / \H(B_2)\leq R_2 $\footnote{due to the assumptions $ q_1 \leq q_2 $ and $ \H(B_2)/\H(B_1) <\epsilon  $, the equality $ \hat{\alpha}_1 + \hat{\alpha}_2 = \mu(B_1) + \mu(B_2) $, and equations \eqref{lebesgue2} and \eqref{steps-starred},  one can take $ R_2 =   \rho(x_1) + \rho(x_2)+ 2  $ as a rough estimate. }.
	Let also $ r_2 $ be such that $ \H(B_2)/\H(B_1) <\epsilon  $ and  $ \hat{\alpha}_2/\H(B_1) <\epsilon  $. Due to  equation   \eqref{lebesgue2}, there holds $ |\mu(B_m)/\H(B_m)-\rho(x_m)|<\epsilon,\ m=1,2 $. Dividing \eqref{finally} through by $ \H(B_1) $ for such a choice of $ r_2 $ gives:
	\begin{equation}\label{finally-divided}
	\begin{aligned}
	 & \c(\rho(x_1)-\epsilon)^{1+s/d}  
	+    (\rho(x_1)-\epsilon )  (q_1-\epsilon)    \\
	\boldsymbol{\leq} & \frac{ \c }{(1-\epsilon)^{s/d} }  \left(\frac{\hat{\alpha}_1}{\H(B_1)}\right)^{1+s/d}    +    \left(\frac{\hat{\alpha}_1}{ \H(B_1)  }\right)  (q_1+\epsilon) \\ 
	& \qquad +\left(\frac{\hat{\alpha}_2}{\H(B_1)}\right) \frac{ \c }{(1-\epsilon)^{s/d} }  \left(\frac{\hat{\alpha}_2}{\H(B_2)}\right)^{s/d}  + \left(\frac{\hat{\alpha}_2}{\H(B_1)}\right)(q_2 + \epsilon) \\
    \boldsymbol{\leq} &   \frac{ \c }{(1-\epsilon)^{s/d} }  { (\rho(x_1)+\epsilon)^{1+s/d}}  +(\rho(x_1)+\epsilon )  (q_1+\epsilon) \\  
	& \qquad +  \epsilon \left(\frac{ \c }{(1-\epsilon)^{s/d} }  { (\rho(x_2)+\epsilon)^{1+s/d}}  +(\rho(x_2)+\epsilon )  (q_2+\epsilon)\right),\\
	\end{aligned}
	\end{equation}
	Finally, because $ \epsilon > 0 $ was arbitrary and the function $ \c t^{1+s/d} + q(x_1)t,\ t\geq 0 $ is monotone,  inequalities \eqref{finally-divided} yield by the above discussion
	\begin{equation}
	 \rho(x_1) =  \lim_{r_1\to 0} \frac{\hat{\alpha}_1}{\H(B_1)}.
	\end{equation}
	We could similarly fix the ball $ B_2 $ first and ensure $ \H(B_1)/\H(B_2) <\epsilon  $, taking $ r_2\to 0 $ afterwards, thus also
	\begin{equation}
	\rho(x_2) =  \lim_{r_2\to 0} \frac{\hat{\alpha}_2}{\H(B_2)}.
	\end{equation}
	In conjunction with \eqref{steps-starred} the last two equations give
	\begin{equation}\label{uniform_rho}
	\frac{ q_2 -  q_1  }{\c (1+s/d) } = \rho(x_1)^{s/d}   - \rho(x_2)^{s/d}
	\end{equation}
	for $ \H\times\H $-a.e. pair $( x_1,x_2 )\in A\times A $. Due to the normalization property $ \int_A \rho (x) \d \H =1 $ and the definition of $ L_1 $ in \eqref{L}, 
	\begin{equation}\label{solution}
	\rho(x) = \left(\frac{L_1 - q(x) }{\c (1+s/d) }\right)^{d/s}_+  \qquad \H\text{-a.e.}
	\end{equation}
	which coincides with the density in formula \eqref{weak_lim}.
	
	\textbf{(\textit{iii}).} Finally, we turn to the case when the function $ q $ need not be bounded above. Consider
	\[q_{ \scalebox{0.6}{\ensuremath{C}}}(x):= \begin{cases} q(x), &q(x)\leq C, \\
	C,\ &\text{otherwise}. \end{cases}\]
	Recall that $ A(C)= \{ x\in A : q(x)\leq C \} $ is a $ d $-rectifiable set as a closed subset of $ A $. The Theorem \ref{thm_weak_lim} is therefore applicable to each function $ \qc$ if seen as defined on  $A(C) $. By Remark \ref{L_is_finite}, the value of $ L_1 $ is finite. For all $C\geq L_1$, 
	\begin{equation}\label{prf_qc_large}
	\supp(\mu_A^{\qc})\cap \{x:q(x)>C\}=\emptyset. 
	\end{equation}
	Inequality $ \qc(x)\leq q(x)$ for all $ x\in A $ implies
	
	\[ \mathcal{E}^{\qc}_{s,d} (A,N)\leq \E (A,N), \quad N\geq 2, \]
	so $ S(\qc,A )\leq \g(A) $. On the other hand, due to set inclusion,
	\[\g(A)\leq \G (A)\leq \G(A(C))= S(\qc,A(C))= S(\qc,A), \]
	where the last two equalities follow from Theorem \ref{thm_weak_lim} applied to the function $ \qc $ and sets $ A(C)$ and $A $ respectively, and equation \eqref{prf_qc_large}. To summarize, $ \g(A)=\G(A)=S(q,A) $.
	
	Let now $ \{{\omega}_N\}_{N\geq2} $ be a sequence satisfying \eqref{asymp_en_min}. Fix a $ C>L_1 $.  Because $ \qc(x)\leq q(x)\ $ for all $ x\in A $ and $ S(q,A)=S(\qc,A) $, the sequence $ \{{\omega}_N\}_{N\geq2} $ is also asymptotically $ (s,d,\qc) $-energy minimizing. Then by Theorem \ref{thm_weak_lim} this sequence converges weak$ ^* $ to $ \d\mu_A^{\qc} $, and it remains to observe that for $ C >L_1 $ it holds $ \d\mu_A^{\qc} =  \d\mu_A^{q} $, where the two measures are defined in equation \eqref{L}.
\end{proof}

\begin{proof}[\textbf{Proof of Theorem \ref{thm_recovery_bounded}.}]
	The desired result is an immediate application of Theorem \ref{thm_weak_lim} since using equation \eqref{L} for the external field from \eqref{q_definition} gives $ L_1=0 $, so the asymptotic distribution is indeed \eqref{density_bounded}.
\end{proof}

\begin{proof}[\textbf{Proof of Proposition \ref{prop_stability}.}]
	We have
	\[   q'(x) = (1+\Delta)q(x). \]
	According to \eqref{L}, the equation 
	\begin{equation}\label{l_equation}
	\int \left(\frac{l-q'(x)}{M_{s,d}}\right)^{d/s}_+\d\H(x) =1 
	\end{equation}
	for variable $ l $ has the unique solution $ l= L'_1 $.
	Using \eqref{q_definition}, it can be rewritten  as
	\begin{equation}\label{l_property}
	\int \left( \rho^{s/d} +  \frac{ \Delta \rho^{s/d} + l }{M_{s,d}}\right)^{d/s}_+\d\H(x)=   1,  
	\end{equation}
	which, in view of $ \int_A \rho \d\H =1  $ and monotonicity of the function $ (\cdot)_+ $, shows that the solution of \eqref{l_equation} satisfies $ |l|\leq |\Delta|\, \|\rho\|_\infty^{s/d} $, that is,
	\[ |L_1'|\leq |\Delta|\, \|\rho\|_\infty^{s/d}. \]
	We will therefore write $ L_1' = \kappa \Delta $ with $ |\kappa|\leq \|\rho\|_\infty^{s/d} $.

	Let us now estimate the difference between densities $ \rho' = \d\mu_A^{q'} / \d\H  $ and $\rho = \d\mu_A^{q}/ \d \H $ in terms of $ \Delta $. Factor out $ \rho^{s/d} $ from the parentheses in \eqref{l_property} and observe that for $ \Delta < {M_{s,d}}/  \left({1+(\|\rho\|_\infty \delta^{-1} )^{s/d} }\right) $ the expression inside is nonnegative, which allows to expand it up to $ o(\Delta) $:
	\[
	\begin{aligned}
	  \rho' =   \rho\left( 1 +  \frac{ \Delta (1 + \kappa /\rho^{s/d} ) }{M_{s,d}}\right)^{d/s}  =  \rho + \Delta  \frac{ d(1 + \kappa /\rho^{s/d} ) }{sM_{s,d}}  + o(\Delta), \qquad \Delta \to 0. 
	\end{aligned}
	\]
\end{proof}

\subsection{Proofs of separation and covering properties}
To obtain point separation results we use techniques of \cite{KuSa1998}, \cite{MEbook}. 
\begin{proof}[\textbf{Proof of Lemma \ref{separation_lemma}.}]	
Fix an $ x\in\hat{\omega}_N $. Because the minimal value of energy $\e(\omega_N)$ is attained for $\hat{\omega}_N $, one must have
\begin{equation}\label{minimality_x^N_i}
U(x,\hat{\omega}_N) \leq U(z,\hat{\omega}_N), \quad z\in A,  
\end{equation}
where $ U(\cdot, \hat{\omega}_N) $ is defined in \eqref{variable_point}. According to Frostman's lemma, \cite[Theorem 8.8]{mattila1999geometry}, for the set $ A $ there exists a positive Borel measure $ \lambda $ satisfying $ \lambda(A) > 0 $ and  such that for all $ x \in \R^p $ and $ R > 0 $,
\begin{equation}\label{semi-d-regular}
\lambda\big( B(x,R)\cap A \big) \leq  R^d.
\end{equation}
By continuity of measure $ \lambda $ from below there exists a positive constant $ H $ for which $ \lambda[A(H)] \geq  2\lambda(A)/3 $; this constant then depends on $ A $  and $ q $, $ H = H(A,q) $. Observe that when $ q $ is bounded from above, $ H $ can be chosen equal to its upper bound. Let $ r_0:= ({\lambda(A)}/{3N} )^{1/d}   $. Consider the set
\[D_x:= A(H)\setminus\bigcup_{\substack{y\in \hat{\omega}_N:\\ y\neq x  }} B(y,r_0). \]
From \eqref{semi-d-regular}:
\[\lambda(D_x)\geq 2\lambda(A)/3-\sum_{\substack{y\in \hat{\omega}_N:\\ y\neq x  }} \lambda \left(A\cap B(y,r_0)\right)\geq \lambda(A)/3.\]
Averaging $ U(z,\hat{\omega}_N) $ on $ D_x $ and taking into account \eqref{minimality_x^N_i} yields
\begin{align}
	U(x,\hat{\omega}_N) \leq& \lambda(D_x)^{-1}\int_{D_x}  U(z,\hat{\omega}_N) d\lambda(z) \nonumber \\ 
	\leq & \frac{3}{\lambda(A)} \left(\sum_{\substack{y\in \hat{\omega}_N:\\ y\neq x  }} \int\limits_{A\setminus B(y,r_0)}    |z-y|^{-s}  \d\lambda(z) + \frac{\T(N)}{N} \int_A  q(z) \d\lambda(z)\right).   \label{integral_line_eq}
\end{align}
Denote $ R_0:= \diam A $. For the integrals in the sum \eqref{integral_line_eq}, use \eqref{semi-d-regular} again :
\begin{equation}\label{I_integral_estimate}
\begin{aligned}
I(y,r):= &\int\limits_{A\setminus B(y,r)}  \negthickspace\negthickspace  |z-y|^{-s}\d\lambda(z) = \int\limits_{0}^{r^{-s}}  \lambda \left\{z\in A\setminus B(y,r): |z-y|^{-s}>t\right\}\d t \\
\leq & \lambda(A)R_0^{-s} + \int\limits_{R_0^{-s}}^{r^{-s} } \lambda\left[ A\cap B(y,t^{-1/s}  ) \right]  \d t\\ \leq & \begin{cases} r^{d-s}\, s/(s-d), &s>d,  \\   (1+d\log(R_0/r)) , &s=d. \end{cases}
\end{aligned} 
\end{equation}
This estimate is independent of $ y $. Using the definition of $ r_0 $, in the case $ s>d $:
\[\begin{aligned}
U(x,\hat{\omega}_N)\leq&  \frac{3}{\lambda(A)} \left( NI(y,r_0) +N^{s/d} \int_{A}  q(z)\d\lambda(z) \right) \\
\leq &\frac{3}{\lambda(A)}\left( \frac{Ns}{r_0^{s-d}(s-d)} +\lambda(A)H N^{s/d}    \right)\\
= &  \left(\frac{s}{s-d}\left( \frac{3}{\lambda(A) } \right)^{s/d} +3H \right)N^{s/d}= C(A,s,d,q)N^{s/d}.
\end{aligned} \]
Similarly, for $ s=d $,
\[ \begin{aligned}
U(x,\hat{\omega}_N)\leq &\frac{3}{\lambda(A)}   N(1+d\log(R_0/r_0))+ 3H N\log N  \\
\leq &\, C(A,d,q) N\log N = C(A,d,q)N\log N.
\end{aligned} \]
This proves the desired statement. 
\end{proof}
\begin{proof}[\textbf{ Proof of Corollary \ref{restrict-minimizers}}]
	It is immediate from Lemma \ref{separation_lemma} and nonnegativity of $ q $ that each $ x\in\hat{\omega}_N $ satisfies $ x\in A(C(A,s,d,q)) $ with the constant taken from \eqref{separation_variable_x}.
\end{proof}
\begin{proof}[\textbf{Proof of Theorem \ref{separation}.}]
Let $ |x -y|=\delta(\hat{\omega}_N ), \ x,y \in \hat{\omega}_N $. From Lemma \ref{separation_lemma}, for $ N\geq 2 $,
\[ \begin{aligned}
\delta(\hat{\omega}_N)^{-s}& =|x-y|^{-s} \leq U(x,\hat{\omega}_N)   \leq U(x,\hat{\omega}_N)\leq C(A,s,d,q) \begin{cases}
N^{s/d} & s >d; \\
N\log N    & s=d, 
\end{cases}
\end{aligned}  \]
 which implies the theorem.
\end{proof}
Similarly to the function \eqref{variable_point}, for an $ r>0 $ and $ y\in A $ let
\begin{equation}\label{u-local}
U_r(x,\omega_N):= \sum_{y\in \omega_N(x,r)}  |y-x|^{-s},
\end{equation}
where $ \omega_N(x,r):=\{y\in\omega_N: y\in B(x,r)\} $ for a fixed sequence of discrete configurations $ \{ {\omega}_N \}_{N\geq2} $.
\begin{lemma}\label{lemma_pre-covering} Let $ s> d $. Assume that $ A\subset\R^{p} $ satisfies $ \mathcal{H}_d(A)>0 $, is compact and $ d $-regular with respect to $ \lambda $. Let  $ q\in L^1 (A,\lambda) $ be a nonnegative lower semi-continuous function and $ \{\hat{\omega}_N\}_{N\geq2} $ be a sequence of $ (s,d,q) $-energy  minimizers. If for a point $ x\in A $ and some $ r>0 $,
\begin{equation}\label{lemma_lower_bound}
U_r(x,\hat{\omega}_N)\geq C  N^{s/d}, \quad N\geq 2, 
\end{equation}
then 
\[ \dist(x,\hat{\omega}_N)\leq \widetilde{C}(C,A,\lambda,s,d) N^{-1/d} , \quad N\geq 2. \]
\end{lemma}
\begin{proof}[\textbf{Proof.}]
The proof follows the lines of \cite{HaSaWh2012}. By Theorem \ref{separation}, there exists a $ C_1>0 $ such that
 \[ \delta(\hat{\omega}_N)\geq {C_1}/{N^{1/d}},\quad N\geq 2. \]
Considering a subsequence if necessary, one may assume $ \dist(x,\hat{\omega}_N)\geq {C_1}/{2N^{1/d}} $, since otherwise the statement of the Lemma follows immediately.  Consider $ r_0:= \epsilon C_1/N^{1/d},\ 0<\epsilon<1/2 $ and put $ B_y:= A\cap B(y,r_0),\ y\in \hat{\omega}_N(x,r) $ for every $ N\geq 2 $. The collection $ \{B_y\} $ defined in this way consists of disjoint sets. By construction,  then for any $ z\in B_y $,
 \[ |z-x|\leq |z-y|+|y-x|\leq r_0+|y-x| \leq (2\epsilon+1)|y-x|, \quad y\in \hat{\omega}(x,r) \]
 where we used that $ r_0 \leq  2\epsilon\,\dist(x,\hat{\omega}_N) \leq 2\epsilon|y-x|  $.
 As $ A $ is $ d $-regular with respect to $ \lambda $, we obtain from the last equation
 \begin{equation}\label{from_whitehouse}
  |y-x|^{-s}\leq \frac{(2\epsilon+1)^{s}}{\lambda{(B_y)}} \int_{B_y}|z-x|^{-s}\d\lambda(z)\leq\frac{(2\epsilon+1)^{s}}{c_0r_0^d} \int_{B_y}|z-x|^{-s}\d\lambda(z).
 \end{equation} 
 Also, for $ z\in B_y $:
 \[ |z-x|\geq|y-x|-|z-y|\geq (1-2\epsilon)|y-x|\geq(1-2\epsilon)\dist(x,\hat{\omega}_N)=:r_\epsilon, \]
 which implies
 \[ \bigcup_{y\in \hat{\omega}(x,r)}B_y \subset A\setminus B(x,r_\epsilon). \]
We write $ \widetilde{c_0}:=(2\epsilon+1)^s/c_0  $. Summing equations \eqref{from_whitehouse} over $y\in \hat{\omega}(x,r)$ and using \eqref{I_integral_estimate},
 \begin{align*}
U_r(x,\hat{\omega}_N)= &\sum_{y\in \hat{\omega}_N(x,r)}  |y-x|^{-s}\leq   \frac{\widetilde{c_0}}{r_0^d}\sum_{y\in \hat{\omega}_N(x,r)} \int_{B_y}|z-x|^{-s}\d\lambda(z) \leq  \frac{\widetilde{c_0}}{r_0^d} \int\limits_{A\setminus  B(x,r_\epsilon)} |z-x|^{-s}\d\lambda(z)\\
 = & \frac{\widetilde{c_0}}{r_0^d}I(x,r_\epsilon)\leq r_\epsilon^{d-s}C_0 \frac{s \widetilde{c_0}}{(s-d) r_0^d}= N[\dist(x,\hat{\omega}_N)]^{d-s} \frac{C_0s\,(1+2\epsilon)^s(1-2\epsilon)^{d-s} }{c_0C_1^d(s-d)\,\epsilon^d}.
  \end{align*}
 The RHS has the minimal value at   $\epsilon= \frac{d}{4s-2d} <1/2 $, if considered as function of $ \epsilon $. Summarizing,  we have 
 \[U_r(x,\hat{\omega}_N)\leq \hat{C}(A,\lambda,s,d) N [\dist(x,\hat{\omega}_N)]^{d-s}. \]
 Substitution of \eqref{lemma_lower_bound} gives
 \[ \dist(x,\hat{\omega}_N)\leq \left(\frac{\hat{C}(A,\lambda,s,d) N}{U_r(x,\hat{\omega}_N)}\right)^{ {1}/{(s-d)}   }\leq \left(\frac{\hat{C}(A,\lambda,s,d)N}{CN^{s/d} }\right)^{{1}/{(s-d)}} , \]
 which ends the proof.
\end{proof}
Recall that we write $  \zeta_N\sim\xi_N ,\ N\to\infty $ if  $ \zeta_N/\xi_N \to1,\ N \to \infty $.

\begin{lemma}\label{covering_lemma}
Let the assumptions of Theorem \ref{covering} hold. Then 
\begin{equation}\label{covering_radius_prereq}
U_r(x,\hat{\omega}_N)\geq C(A,	h,s,d,q )  N^{s/d}, \qquad N\geq 2.
\end{equation}
\end{lemma}

\begin{proof}[\textbf{Proof.}]
 For a fixed $ \Delta>0 $, choose small enough $ r>0 $, so that  $z\in B(x,r) $ implies $q(z)\leq L_1 -h/2 $ and $ q(z)\geq  q(x)-\Delta $ for all $ x\in A(L_1-h) $. The choice of $ r  $ thus depends on $ q,h,A,\Delta $. Note that by \eqref{L}, $x\in \supp(\mu^q_A) $. Suppose also that $ r $ satisfies $ \H(\partial B(x,r) ) =0 $ (such values of $ r $ are dense because $ \h(A)<\infty $).
As above,  $ \hat{\omega}_N(x,r)=\{y\in\hat{\omega}_N: y\in B(x,r)\} $.  
Using equation \eqref{regular_WRT_measure} and Theorem \ref{thm_weak_lim}, we have:
\begin{equation}\label{B_r_lower}
\begin{aligned}
 \mu_A^q[B(x,r)]=  \int\limits_{B(x,r)} & \left(\frac{L_1-q(z)}{\c(1+s/d)}\right)^{d/s}   \d\H(z)\geq\\
 & \geq \left(\frac{h}{2\c(1+s/d)}\right)^{d/s}c_0 r^d =: c(h,A,s,d)r^d.
\end{aligned}
\end{equation}
For any $ N\geq 2 $, if $ x\in \hat{\omega}_N $ there is nothing to prove. Otherwise, as $\hat{\omega}_N$ is an optimal configuration, from \eqref{minimality_x^N_i} for every $y\in \hat{\omega}_N(x,r)$
\[\begin{aligned}
U_r(x,\hat{\omega}_N) + N^{s/d} { q(x)} &\boldsymbol{\geq} |x-y|^{-s}  +\sum_{\substack{z\in  \hat{\omega}_N(x,r):\\ z\neq y}} |z-y|^{-s} + N^{s/d} {q(y)}  + \\
&  \qquad  + \sum_{\substack{z\in\hat{\omega}_N:\\ z\notin  \hat{\omega}_N(x,r)}}\left(|z-y|^{-s}-|z-x|^{-s} \right) \\
& \boldsymbol{\geq} |x-y|^{-s}  +\sum_{\substack{z\in  \hat{\omega}_N(x,r):\\ z\neq y}} |z-y|^{-s} + N^{s/d}{ q(y)}-Nr^{-s}.
\end{aligned} \]
Summing over all $y\in \hat{\omega}_N(x,r)$,
\begin{equation}\label{asymptotics_covering}
\begin{aligned}
(\#\hat{\omega}_N(x,r)-1) & U_r(y,\hat{\omega}_N) \geq 
\\
\geq & \sum_{\substack{y,z\in  \hat{\omega}_N(x,r):\\ z\neq y}} |z-y|^{-s} + N^{s/d} \sum_{\substack{y \in  \hat{\omega}_N(x,r)}}  ( q(y)  - q(x) )-N^2r^{-s}\\ 
\geq  & \sum_{\substack{y,z\in  \hat{\omega}_N(x,r):\\ z\neq y}} |z-y|^{-s} - N^{s/d} {\Delta \#\hat{\omega}_N(x,r) }-N^2r^{-s}.
\end{aligned}
\end{equation}
Since $ \H(\partial B(x,r) ) =0 $, there holds  $\lim \#\hat{\omega}_N(x,r)/N = \mu^q_A [B(x,r)] $, $ N\to\infty $. Using \eqref{B_r_lower}  and the Lemma \ref{lower-optimum} for the single set $ B(x,r)  $ with $ q(\cdot) \equiv 0 $, we conclude
\[\begin{aligned}
\lim\inf_{N\to \infty } N^{-1-s/d} \sum_{\substack{y,z\in  \hat{\omega}_N(x,r):\\ z\neq y}} |z-y|^{-s} \geq  \frac{\H[B(x,r)]}{\c^{d/s}} \left(\frac{h}{  2(1+s/d) }\right)^{1+d/s} .
\end{aligned} \] 
Since  $N^2r^{-s}=o(N^{1+s/d})$, dividing \eqref{asymptotics_covering} by $ \#\hat{\omega}_N(x,r) \sim N\mu^q_A [B(x,r)]  $ for $ N\to\infty $  gives
\begin{equation}\label{r^-1}
 \begin{aligned}
 U_r & (x,\hat{\omega}_N) \geq  \\
 & \geq  N^{s/d}\left(  \frac{\H[B(x,r)]}  {\mu^q_A [B(x,r)]\c^{d/s}} \left(\frac{h}{  2(1+s/d)}\right)^{1+d/s}-\Delta- N^{1-s/d}\frac{r^{-s}}{\mu^q_A [B(x,r)]}  \right) \\
 & \geq  N^{s/d} \left(\frac{(h/2)^{1+d/s}}{(1+s/d)L_1^{d/s}}-\Delta-N^{1-s/d} \frac{r^{-s-d}}{c(h,A,s,d)} \right), \qquad N\to\infty,
 \end{aligned} 
\end{equation}
where we used 
\[ \mu_A^q[B(x,r)]\leq \left(\frac{L_1}{\c(1+s/d)}\right)^{d/s}\H[B(x,r)]. \]
If we put $ \Delta = \frac{(h/2)^{1+d/s}}{2(1+s/d)L_1^{d/s}}$, the inequality \eqref{r^-1} implies that  there exists a constant $ C = C(A,h,s,d,q ) $ for which
\begin{equation}\label{covering_constant}
U_r(x,\hat{\omega}_N)\geq C(A,h,s,d,q )  N^{s/d}, \quad N\in\N.
\end{equation}
\end{proof}
\begin{proof}[\textbf{Proof of Theorem \ref{covering}.}] Follows from Lemma \ref{lemma_pre-covering} and Lemma \ref{covering_lemma}.
\end{proof}

\bibliographystyle{abbrv}
\bibliography{../energy}
\end{document}